\documentclass[11pt,reqno]{amsart}
\usepackage{amsmath, amssymb}

\newtheorem{Lemma}{Lemma}[section]
\newtheorem{Theorem}[Lemma]{Theorem}
\newtheorem{Corollary}[Lemma]{Corollary}

\theoremstyle{definition}
\newtheorem{Remark}[Lemma]{Remark}

\newtheorem*{Remark*}{Remark}

\newcommand{\thref}[1]{Theorem \ref{#1}}
\newcommand{\leref}[1]{Lemma \ref{#1}}

\newcommand{\reref}[1]{Remark \ref{#1}}

\newcommand{\coref}[1]{Corollary \ref{#1}}
\newcommand{\seref}[1]{Section \ref{#1}}
\newcommand{\ssref}[1]{Subsection \ref{#1}}

\numberwithin{equation}{section}

\newcommand{\Cset}{\mathbb{C}}
\newcommand{\Nset}{\mathbb{N}}

\newcommand{\Zset}{\mathbb{Z}}
\newcommand{\Tset}{{\mathbb T}}

\newcommand{\At}{{\tilde{A}}}
\newcommand{\cEt}{{\tilde{\mathcal{E}}}}
\newcommand{\Ct}{{\tilde{C}}}

\newcommand{\It}{{\tilde{I}}}
\newcommand{\Eh}{{\hat{E}}}
\newcommand{\Eht}{\tilde{{\hat{E}}}}

\newcommand{\G}{{\Gamma}}
\newcommand{\Gt}{{\tilde{\Gamma}}}
\newcommand{\cK}{{\mathcal K}}
\newcommand{\cKt}{{\tilde{\mathcal K}}}

\newcommand{\Kt}{\tilde{K}}
\newcommand{\phit}{\tilde{\phi}}
\newcommand{\psit}{\tilde{\psi}}
\newcommand{\phitr}{\overleftarrow{\tilde{\phi}}}
\newcommand{\psitr}{\overleftarrow{\tilde{\psi}}}
\newcommand{\Phit}{\tilde{\Phi}}
\newcommand{\Psit}{\tilde{\Psi}}
\newcommand{\phir}{\overleftarrow{\phi}}
\newcommand{\psir}{\overleftarrow{\psi}}
\newcommand{\Phir}{\overleftarrow{\Phi}}

\newcommand{\Phitr}{\overleftarrow{\tilde{\Phi}}}
\newcommand{\Phip}{\Phi^{p}}
\newcommand{\Phiq}{\Phi^{q}}
\newcommand{\phib}{\bar{\phi}}
\newcommand{\Span}{\mathrm{span}}
\newcommand{\pb}{\bar{p}}
\newcommand{\qb}{\bar{q}}
\newcommand{\pr}{\overleftarrow{p}}
\newcommand{\qr}{\overleftarrow{q}}

\newcommand{\hr}{\overleftarrow{h}}

\newcommand{\zb}{\bar{z}}
\newcommand{\wb}{\bar{w}}
\newcommand{\Pinm}{\Pi^{n,m}}
\newcommand{\PiNM}{\Pi^{N,M}}

\newcommand{\rank}{\mathrm{rank}}

\newcommand{\revz}[1]{\overleftarrow{{#1}^{\theta}}}

\newcommand{\cL}{\mathcal{L}}
\newcommand{\cLz}{\mathcal{L}^{\theta}}

\newcommand{\cLp}{\mathcal{L}_{p}}
\newcommand{\cLq}{\mathcal{L}_{q}}
\newcommand{\phiz}[1]{{\phi_{#1}^{\theta}}}

\newcommand{\Cz}[1]{C_{#1}^{\theta}}
\newcommand{\cz}[1]{c_{#1}^{\theta}}
\newcommand{\Ut}{\tilde{U}}
\newcommand{\ut}{\tilde{u}} 
\newcommand{\Vt}{\tilde{V}}
\newcommand{\nn}{\nonumber}

\begin{document}
\title[Fej\'er-Riesz factorizations and bivariate polynomials]{Fej\'er-Riesz factorizations and the structure of bivariate polynomials
orthogonal on the bi-circle}

\author[J.~Geronimo]{Jeffrey~S.~Geronimo}
\thanks{JG is supported in part by Simons Foundation Grant \#210169.}
\address{JG, School of Mathematics, Georgia Institute of Technology,
Atlanta, GA 30332--0160, USA}
\email{geronimo@math.gatech.edu}

\author[P.~Iliev]{Plamen~Iliev}
\thanks{PI is supported in part by NSF Grant \#0901092.}
\address{PI, School of Mathematics, Georgia Institute of Technology,
Atlanta, GA 30332--0160, USA}
\email{iliev@math.gatech.edu}

\begin{abstract}
We give a complete characterization of the positive trigonometric polynomials $Q(\theta,\varphi)$ on the bi-circle, which can be factored as $Q(\theta,\varphi)=|p(e^{i\theta},e^{i\varphi})|^2$ where $p(z, w)$ is a polynomial nonzero for $|z|=1$ and $|w|\leq1$. 
The conditions are in terms of recurrence coefficients associated with the polynomials in lexicographical and reverse lexicographical ordering orthogonal with respect to the weight $\frac{1}{4\pi^2Q(\theta,\varphi)}$ on the bi-circle. We use this result to describe how specific factorizations of weights on the bi-circle can be translated into identities relating the recurrence coefficients for the corresponding polynomials and vice versa. In particular, we characterize the Borel measures on the bi-circle for which the coefficients multiplying the reverse polynomials associated with the two operators: multiplication by $z$ in lexicographical ordering and multiplication by $w$ in reverse lexicographical ordering vanish after a particular point. This can be considered as a spectral type result analogous to the characterization of the Bernstein-Szeg\H{o} measures on the unit circle.
\end{abstract}

\date{June 6, 2012}

\maketitle

\section{Introduction}

The factorization of positive polynomials as a sum of squares of polynomials or rational functions is an important problem in mathematics and led Hilbert to pose his 17th problem which was solved by Artin. In the case of trigonometric polynomials one of the simplest factorization results is 
the lemma of Fej\'er-Riesz which states that every positive trigonometric polynomial $Q_n(\theta)$ of degree $n$ can be written as $Q_n(\theta)=|p_n(e^{i\theta})|^2$ where $p_n(z)$ is a polynomial of degree $n$ in $z$. This result has been useful for the trigonometric moment problem, orthogonal polynomials, wavelets, and signal processing. 

Extensions of this result to the multivariable case cannot be generic as a simple degree of freedom calculation on the coefficients shows. Recently
\cite{GW1} these results have been extended to two variable factorizations 
\begin{equation}\label{1.1}
Q_{n,m}(\theta,\varphi)=|p_{n,m}(e^{i\theta},e^{i\varphi})|^2
\end{equation}
where $n$ and $m$ are the degrees of $Q_{n,m}$ in $\theta$ and $\varphi$, respectively, in the case when $p_{n,m}(z,w)$ is a polynomial of degree $n$ in $z$ and $m$ in $w$ which is nonzero for $|z|\leq 1$ and $|w|\leq1$. This augments results obtained earlier by Kummert \cite{Ku} (see also Ball \cite{B}), Cole and Wermer \cite{CW}, and Agler and McCarthy \cite{AM} (see also Knese \cite{Kn}).  In particular, using the results of Knese \cite{Kn1} it is easy to see that, except for certain special cases, the polynomials $p_{n,m}$ in \eqref{1.1} cannot be associated with the distinguished varieties defined by  Agler and McCarthy \cite{AM}. Some extensions to more than two variables of the above  results have also recently been  obtain by Grinshpan et al \cite{GKVW}, Bakonyi and Woerdeman \cite{BW}, and Woerdeman \cite{W}. In this paper we extend the results in \cite{GW1} in a different direction.
We characterize completely positive trigonometric polynomials $Q_{n,m}(\theta,\varphi)$, which can be factored as in \eqref{1.1} where $p_{n,m}(z, w)$ is a 
polynomial which is nonzero for $|z|=1$ and $|w|\leq1$. The conditions can be written in a relatively simple form if we use the orthogonal polynomials in lexicographical and reverse lexicographical ordering introduced in \cite{GW2} with respect to the weight $\frac{1}{4\pi^2Q_{n,m}(\theta,\varphi)}$ on the bi-circle. More precisely, in \thref{th2.4} we prove that \eqref{1.1} holds if and only if certain matrices (which represent recurrence coefficients) $\cK_{n,m}$, $\cK^1_{n,m}$, $\Gt_{n,m}$, $\Gt^1_{n,m}$ satisfy the equations
\begin{equation}\label{1.2}
\cK_{n,m}[\Gt^1_{n,m}\Gt^{\dagger}_{n,m}]^{j}(\cK^1_{n,m})^T=0, \text{ for }j=0,1,
\dots,n-1.
\end{equation} 
There are two important cases when equation \eqref{1.2} holds:
\begin{itemize}
\item[(i)] The case when $\cK_{n,m}=0$ characterizes the {\em stable factorizations} of $Q_{n,m}$ discussed in \cite{GW1} (i.e. equation \eqref{1.1} holds with a polynomial $p_{n,m}(z,w)$ which is nonzero for $|z|\leq 1$ and $|w|\leq1$).
\item[(ii)] The case when $\cK^{1}_{n,m}=0$ characterizes the {\em anti-stable factorizations} of $Q_{n,m}$. In this case \eqref{1.1} holds with a polynomial $p_{n,m}(z,w)$ such that $z^np_{n,m}(1/z,w)\neq0$ for $|z|\leq 1$ and $|w|\leq1$).
\end{itemize}
We derive several corollaries of the above result which are of independent interest. For instance, 
we characterize the Borel measures on the
bi-circle for which the recurrence coefficients $\Eh_{k,l}$, $\Eht_{k,l}$ multiplying the reverse polynomials associated with the two operators: multiplication by $z$ in lexicographical ordering and multiplication by $w$ in reverse lexicographical vanish after a particular point, see \thref{th2.10}. This can be considered as a spectral theory type result analogous to the characterization of  Bernstein-Szeg\H{o} measures on the circle.  We also show that in this case the space of orthogonal polynomials can be decomposed as an appropriate direct sum of two sets of orthogonal polynomials associated with the stable and the anti-stable factorizations described above, see \thref{th2.7}. 

The paper is organized as follows. In \seref{se2} we introduce the notation used throughout the paper including the recurrence formulas and state the main theorems.  In \seref{se3} some preliminary results are proved and certain relations among the recurrence coefficients developed in \cite{GB} and their consequences are discussed. In \seref{se4} we prove the first main theorem which yields the factorizations \eqref{1.1} with $p_{n,m}(z,w)$ nonzero for $|z|=1$ and $|w|\leq1$. In the forward direction, we use the Gohberg-Semencul formula, parametric and matrix-valued orthogonal polynomials to show that if \eqref{1.1} holds, then the recurrence coefficients $\Eh_{k,l}$ for the polynomials in lexicographical ordering associated with the weight $\frac{1}{4\pi^2 Q_{n,m}(\theta,\varphi)}$ vanish after a particular point. This leads to \eqref{1.2}. The heart of the proof in the opposite direction is based on a very subtle decomposition of the space of polynomials in the reverse lexicographical ordering as the  sum of two subspaces possessing a lot of extra orthogonality properties. Using this decomposition, we construct an appropriate rotation on the space of polynomials which gives the polynomial $p_{n,m}(z,w)$ satisfying \eqref{1.1}. All these constructions are missing in the stable case: the space decomposition is trivial (one of the subspaces is empty) and the rotation is simply the identity transformation. Thus, in our construction, the polynomial $p_{n,m}(z,w)$ is no longer the first column of the inverse of the Toeplitz matrix associated with the trigonometric moments, but instead is a linear combination of the columns in the first block column of this matrix. One can use also the general theory of Helson and Lowdenslager \cite{HL} and the constructions in Delsarte et al \cite{DGK2} to obtain factorizations of positive functions $Q(\theta,\varphi)$ on the bi-circle. Note, however, that their approach works in a rather general setting and will provide (in general) non-polynomial factorizations of $Q_{n,m}$, even when
  \eqref{1.1} holds with a polynomial $p_{n,m}(z,w)$. In \seref{se5} we prove all remaining statements and corollaries. In \seref{se6} some examples are presented as illustrations of the main theorems.

\section{Statement of results}\label{se2}
\subsection{Basic notations}\label{ss2.1}
We denote $\Tset=\{z\in\Cset:|z|=1\}$ the unit circle and 
$$\Tset^2=\{(z,w):|z|=|w|=1\},$$
the bi-circle (torus) in $\Cset^2$. Throughout the paper, we will use the 
parametrization 
$z=e^{i\theta}$ and $w=e^{i\varphi}$, where $\theta,\varphi\in[-\pi,\pi]$.

We consider moment matrices associated with the
lexicographical ordering which is defined by
$$
(k,\ell)<_{\rm lex} (k_1,\ell_1)\Leftrightarrow k<k_1\mbox{ or }
(k=k_1\mbox{ and } \ell<\ell_1),
$$
and the reverse lexicographical ordering defined by
$$
(k,\ell)<_{\rm revlex} (k_1,\ell_1)\Leftrightarrow
(\ell,k)<_{\rm lex} (\ell_1,k_1).
$$
Both of these orderings are linear orders and in addition they satisfy
$$
(k,\ell)<(m,n)\Rightarrow (k+p,\ell+q)<(m+p,n+q).
$$
Let $\Pinm$ denote the bivariate Laurent linear subspace
$\Span \{z^kw^l,\, -n\le k\le n,\,-m\le l\le m\}$ and let
$\cL$ be a linear functional defined on $\Pinm$ such that
$$
\cL(z^{-k}w^{-l})=c_{k,l}= \overline{\cL(z^k w^l)}.
$$
We will call $c_{k,l}$ the $(k,l)$ moment of $\cL$ and
$\cL$ a moment functional.  If we form the
$(n+1)(m+1)\times(n+1)(m+1)$ matrix $C_{n,m}$ for $\cL$ in
the lexicographical ordering then it has the special block Toeplitz form
\begin{equation}\label{2.1}
C_{n,m} = \left[
\begin{matrix}
C_0 & C_{-1} & \cdots & C_{-n}
\\
C_1 & C_0 & \cdots & C_{-n+1}
\\
\vdots &  & \ddots & \vdots \\ 
C_n & C_{n-1} & \cdots & C_{0}
\end{matrix}
\right],
\end{equation}
where each $C_k$ is an $(m+1)\times(m+1)$ Toeplitz matrix as
follows
\begin{equation}\label{2.2}
C_k=\left[
\begin{matrix}
c_{k,0}& c_{k,-1} & \cdots & c_{k,-m}
\\
\vdots & &\ddots & \vdots \cr c_{k,m} &  & \cdots &
c_{k,0}
\end{matrix}
\right],\qquad k=-n,\dots, n.
\end{equation}
Thus $C_{n,m}$ has a doubly Toeplitz structure. If the reverse
lexicographical ordering is used in place of the lexicographical
ordering we obtain another moment matrix $\Ct_{n,m}$ where
the roles of $n$ and $m$ are interchanged.
We say that the moment functional $\cL:\Pinm\to\Cset$
is positive  if
\begin{equation}\label{2.3}
\cL \left[p(z,w)\pb(1/z,1/w)\right]>0 
\end{equation}
for every nonzero polynomial $p(z,w)\in\Pinm\cap\Cset[z,w]$. Here and later we set 
$\pb(z,w)=\overline{p(\bar{z},\bar{w})}$. It
follows from a simple quadratic form argument that $\cL$ is
positive  if and only if its
moment matrix $C_{n,m}$ is positive definite.
We now perform the Gram-Schmidt procedure on the monomials using the 
lexicographical ordering. The study of orthogonal  polynomials on the bi-circle 
with this ordering was begun by Delsarte et al. \cite{DGK1} and extended in 
\cite{GW2}. Given a positive definite linear functional $\cL:\PiNM\to\Cset$ we 
perform the Gram-Schmidt procedure using the lexicographical ordering on the 
spaces $\Span\{z^kw^l:0\leq k\leq n,\,0\leq l\leq m\}$ where $n\leq N$, $m\leq N$. Thus we define the orthonormal 
polynomials
$\phi_{n,m}^s(z,w),\ 0\le n\le N,\, 0\le m\le M, \, 0\le s\le m,$
by the equations
\begin{equation}\label{2.4}
\begin{split}
&\cL(\phi_{n,m}^s(z,w) z^{-k}w^{-l})=0, \quad  0\le k<n\ \text{ and }\
0\le l\le m\text{ or }\ k=n \text{ and }\ 0\le l< s,\\
&\cL(\phi_{n,m}^s(z,w)\bar{\phi}_{n,m}^s(1/z,1/w))=1,
\end{split}
\end{equation}
and
\begin{equation}\label{2.5}
\phi_{n,m}^{s}(z,w) = k^{n,s}_{n,m,s} z^n w^s + \sum_{(k,l)<_{\rm
lex}(n,s)} k^{k,l}_{n,m,s}z^kw^l.
\end{equation}
With the convention $k^{n,s}_{n,m,s}>0$, the above equations uniquely
specify $\phi^s_{n,m}$. Polynomials orthonormal with respect to
$\cL$ but using the reverse lexicographical ordering will be
denoted by $\phit^s_{n,m}$. They are uniquely determined by the
above relations with the roles of $n$ and $m$ interchanged.
Set
\begin{equation}\label{2.6}
\Phi_{n,m}(z,w)=\left[\begin{matrix} \phi_{n,m}^{m}\\ \phi_{n,m}^{m-1}\\[-2pt]
\vdots\\ \phi_{n,m}^{0} \end{matrix}\right] =
K_{n,m}\left[\begin{matrix} z^n w^m\\ z^n w^{m-1}\\[-2pt] \vdots\\ 1\end{matrix}\right],
\end{equation}
where the $(m+1)\times(n+1)(m+1)$ matrix $K_{n,m}$ is given by
\begin{equation}\label{2.7}
K_{n,m}=\left[\begin{matrix} k_{n,m,m}^{n,m}&
k_{n,m,m}^{n,m-1}&\cdots & \cdots&\cdots& k_{n,m,m}^{0,0} \\ 0&
k_{n,m,m-1}^{n,m-1}&\cdots & \cdots&\cdots&
k_{n,m,m-1}^{0,0}\\\vdots &\ddots&\ddots&\ddots&\ddots&\ddots
\\0&\cdots& k_{n,m,0}^{n,0}& k_{n,m,0}^{n-1,m}&\cdots &
k_{n,m,0}^{0,0}\end{matrix}\right].
\end{equation}
As indicated above denote
\begin{equation}\label{2.8}
\Phit_{n,m} (z,w)=\left[\begin{matrix} \phit_{n,m}^{n}\\ 
\phit_{n,m}^{n-1}\\ \vdots\\ \phit_{n,m}^{0}\end{matrix}\right]
=\Kt_{n,m}\left[\begin{matrix}w^m z^n \\  w^m z^{n-1}\\[-3pt] \vdots\\ 1
\end{matrix}\right],
\end{equation}
where the $(n+1)\times(n+1)(m+1)$ matrix $\Kt_{n,m}$ is
given similarly to \eqref{2.7} with the roles of $n$ and $m$
interchanged. For the bivariate polynomials $\phi^s_{n,m}(z,w)$ above
we define the reverse polynomials $\phir^s_{n,m}(z,w)$ by the relation
\begin{equation}\label{2.9}
\phir^s_{n,m}(z,w)=z^n w^m\bar{\phi}_{n,m}^s(1/z,1/w).
\end{equation}
With this definition $\phir^s_{n,m}(z,w)$ is again a
polynomial in $z$ and $w$, and furthermore
\begin{equation}\label{2.10}
\Phir_{n,m}(z,w):= \left[\begin{matrix}\phir_{n,m}^{m}\\ 
\phir_{n,m}^{m-1}\\
\vdots\\ \phir_{n,m}^{0} \end{matrix}\right]^T.
\end{equation}
An analogous procedure is used to define $\phitr^s_{n,m}$. We use $M^{m,n}$ to 
denote the space of all $m\times n$ matrices. 
In \cite{GW2} it was shown:
\begin{Theorem}\label{recurrencefor}
Given $\{\Phi_{n,m}\}$ and $\{\Phit_{n,m}\}$, $0\le n\le N$,
$0\le m\le M$, the following recurrence formulas hold:
\begin{subequations}\label{2.11}
\begin{align}
& A_{n,m}\Phi_{n,m} = z\Phi_{n-1,m} - \Eh_{n,m}\Phir_{n-1,m}^T , \label{2.11a}
\\
&\Phi_{n,m}+ A^{\dagger}_{n,m}\Eh_{n,m}(A^T_{n,m})^{-1}\Phir_{n,m}^T=
A^{\dagger}_{n,m}z\Phi_{n-1,m}, \label{2.11b}
\\
&\G_{n,m} \Phi_{n,m} = \Phi_{n,m-1} - {\cK}_{n,m} \Phit_{n-1,m},
\label{2.11c}
\\
&\G_{n,m}^1 \Phi_{n,m} = w \Phi_{n,m-1} - {\cK}^1_{n,m}\Phitr_{n-1,m}^T,
\label{2.11d}
\\
&\Phi_{n,m}=I_{n,m} \Phit_{n,m} + \G^{\dagger}_{n,m} \Phi_{n,m-1}
\label{2.11e},
\\
&\Phir_{n,m}^T=I_{n,m}^1 \Phit_{n,m} + (\G^1_{n,m})^T\Phir_{n,m-1}^T
\label{2.11f},
\end{align}
\end{subequations}
where
\begin{subequations}\label{2.12}
\begin{align}
\Eh_{n,m} & = \langle z\Phi_{n-1,m},\Phir_{n-1,m}^T\rangle=\Eh_{n,m}^T \in
M^{m+1,m+1}, \label{2.12a}\\
A_{n,m} & = \langle z\Phi_{n-1,m},\Phi_{n,m}\rangle \in
M^{m+1,m+1}, \label{2.12b}\\
{\cK}_{n,m} & = \langle \Phi_{n,m-1}, \Phit_{n-1,m}\rangle \in
M^{m,n}, \label{2.12c}\\
\G_{n,m} & = \langle \Phi_{n,m-1}, \Phi_{n,m} \rangle \in M^{m,m+1},\label{2.12d}\\
{\cK}^1_{n,m} & = \langle w \Phi_{n,m-1}, \overleftarrow{\Phit}_{n-1,m}^T\rangle \in M^{m,n},\label{2.12e}\\
\G^1_{n,m} & = \langle w \Phi_{n,m-1}, \Phi_{n,m} \rangle \in
M^{m,m+1}, \label{2.12f}\\
I_{n,m} & = \langle \Phi_{n,m}, \Phit_{n,m}\rangle \in
M^{m+1,n+1}, \label{2.12g}\\
I^1_{n,m} & = \langle \Phir_{n,m}^T, \Phit_{n,m}\rangle \in
M^{m+1,n+1}. \label{2.12h}
\end{align}
\end{subequations}
\end{Theorem}

\begin{Remark}
From now on we adapt the following convention. For every statement (resp. formula) we will refer to the analogous statement (resp. formula) with the roles of $z$ and $w$ exchanged as the tilde analog. For instance, the tilde analog of formula \eqref{2.11a} is 
$\At_{n,m}\Phit_{n,m} = w\Phit_{n,m-1} - \Eht_{n,m}\Phitr_{n,m-1}^T$.
\end{Remark}

Finally, we note that 
\begin{equation*}
\cKt_{n,m}=\cK_{n,m}^{\dagger}\quad\text{ and }\quad \cKt^{1}_{n,m}=(\cK^{1}_{n,m})^{T}.
\end{equation*}

\subsection{Main results}
We say that a polynomial $p(z,w)\in\Cset\left[z,w\right]$ 
is of degree $(n,m)$ where $n$ and $m$ are the minimal nonnegative integers 
such that $p(z,w)\in\Pinm$. We say that the polynomial $p(z,w)$ is {\em stable\/} if it does not vanish for 
$|z|\leq 1$ and $|w|\leq 1$. Similarly, for a trigonometric polynomial 
$Q(\theta,\varphi)=p(e^{i\theta},e^{i\varphi})$, we define the degree as the 
ordered pair $(n,m)$, where $n$ and $m$ are the minimal nonnegative integers 
such that $p(z,w)\in\Pinm$.

We can now state our main results.
\begin{Theorem}\label{th2.3}
For a positive moment functional $\cL$ defined on the space $\Pinm$ 
the following conditions are equivalent:
\begin{itemize}
\item[{\rm{(i)}}] There exists a polynomial $p(z,w)$ of degree at most 
$(n,m)$, nonzero for $ |z|=1$ and $|w|\leq1$, such that 
\begin{equation}\label{2.13}
\cL(z^kw^l)=\frac{1}{4\pi^2}\int\limits_{[-\pi,\pi]^2}\frac{e^{ik\theta}e^{il\varphi}}{|p(e^{i\theta},e^{i\varphi})|^2}\,d\theta\,d\varphi.
\end{equation}
\item[{\rm{(ii)}}] The coefficients $\cK_{n,m}$, $\cK^1_{n,m}$, $\Gt_{n,m}$, 
$\Gt^1_{n,m}$ satisfy
\begin{equation}\label{2.14}
\cK_{n,m}[\Gt^1_{n,m}\Gt^{\dagger}_{n,m}]^{j}(\cK^1_{n,m})^T=0, \text{ for }j=0,1,
\dots,n-1.
\end{equation}
\end{itemize} 
Moreover, if the conditions above hold, we have
\begin{equation}\label{2.15}
\begin{split}
|p(z,w)|^2=&\Phi_{n,m}(z,w)^{T}\overline{\Phi_{n,m}(z,w)}-\Phi_{n,m-1}(z,w)^{T}\overline{\Phi_{n,m-1}(z,w)}\\
=&\Phit_{n,m}(z,w)^{T}\overline{\Phit_{n,m}(z,w)}-\Phit_{n-1,m}(z,w)^{T}\overline{\Phit_{n-1,m}(z,w)}, \\
&\qquad\qquad \qquad\qquad \text{ for }(z,w)\in\Tset^2.
\end{split}
\end{equation}
\end{Theorem}
The polynomial $p(z,w)$ in \thref{th2.3} can be computed from equation \eqref{4.11} in \seref{se4}, which depends on the matrices $\Ut$ and $\Vt$ constructed from $\cK_{n,m}$, $\cK^{1}_{n,m}$, $\Gt_{n,m}$ and $\Gt^{1}_{n,m}$ in \leref{le4.6} and \leref{le4.8}, see \reref{re4.5} for more details.

As an immediate corollary of the above theorem and the maximum entropy principle \cite{BN} we obtain the first 
Fej\'er-Riesz factorization.
\begin{Theorem}[Fej\'er-Riesz I] \label{th2.4}
Suppose that $Q(\theta,\varphi)$ is a strictly 
positive trigonometric polynomial of degree $(n,m)$. Then 
$Q(\theta,\varphi)=|p(e^{i\theta},e^{i\varphi})|^2$ where $p(z,w)$ is a polynomial of degree $(n,m)$ such that $p(z,w)\neq 0$ for $ |z|=1$, $|w|\leq 1$ if and 
only if the coefficients $\cK_{n,m}$, $\cK^1_{n,m}$, $\Gt_{n,m}$, 
$\Gt^1_{n,m}$ associated with the measure 
$\frac{d\theta\, d\varphi}{4\pi^2 Q(\theta,\varphi)}$ on $[-\pi,\pi]^2$ satisfy equation \eqref{2.14}.
\end{Theorem}
Analogous results hold with the roles of $z$ and $w$ and $n$ and $m$ interchanged if the coefficients in the reverse lexicographical ordering satisfy the tilde analogs of equation~\eqref{2.14} (see equation~\eqref{2.17b} below). In the case when both sets of conditions hold we find:
\begin{Theorem}\label{th2.5}
For a positive moment functional $\cL$ defined on the space $\Pinm$ the following conditions are equivalent:
\begin{itemize}
\item[{\rm{(i)}}] There exist stable polynomials $p(z,w)$ and $q(z,w)$ of degrees $(n_1,m_1)$ and $(n_2,m_2)$ with $n_1+n_2\leq n$, $m_1+m_2\leq m$ such that 
\begin{equation}\label{2.16}
\cL(z^kw^l)
=\frac{1}{4\pi^2}\int\limits_{[-\pi,\pi]^2}\frac{e^{ik\theta}e^{il\varphi}}{|p(e^{i\theta},e^{i\varphi})q(e^{-i\theta},e^{i\varphi})|^2}\,d\theta\,d\varphi.
\end{equation}
\item[{\rm{(ii)}}] The coefficients $\cK_{n,m}$,  $\cK^1_{n,m}$, $\G_{n,m}$, $\G^1_{n,m}$, $\Gt_{n,m}$,  $\Gt^{1}_{n,m}$ satisfy
\begin{subequations}\label{2.17}
\begin{align}
&\cK_{n,m}[\Gt^1_{n,m}\Gt^{\dag}_{n,m}]^{j}(\cK^1_{n,m})^T=0, \text{ for }j=0,1,\dots,n-1, \label{2.17a} \\
&\cK_{n,m}^{\dag}[\G^1_{n,m}\G^{\dag}_{n,m}]^{l} \cK^1_{n,m}=0, \text{ for }l=0,1,\dots,m-1. \label{2.17b}
\end{align}
\end{subequations}
\end{itemize} 
\end{Theorem}
As in \thref{th2.3}, given the coefficients in the recurrence formulas, the polynomial $p(z,w)z^{n_2}q(1/z,w)$ can be computed by \eqref{4.11} (see also \reref{re4.5}). 
In view of equation \eqref{2.16}, we say in the rest of the paper that a functional satisfying the equivalent conditions in the above theorem belongs to the {\em splitting case}. \thref{th2.5} can also be recast as a Fej\'er-Riesz factorization.

\begin{Theorem}[Fej\'er-Riesz II] \label{th2.6}
Suppose that $Q(\theta,\varphi)$ is a strictly positive trigonometric polynomial of degree $(n,m)$. Then $Q(\theta,\varphi)=|p(e^{i\theta},e^{i\varphi})q(e^{-i\theta},e^{i\varphi})|^2$ where $p(z,w)$ and $q(z,w)$ are stable polynomials of degrees $(n_1,m_1)$ and $(n_2,m_2)$ respectively, with $n_1+n_2= n$,  $m_1+m_2= m$ if and only if the coefficients $\cK_{n,m}$,  $\cK^1_{n,m}$, $\G_{n,m}$, $\G^1_{n,m}$, $\Gt_{n,m}$,  $\Gt^{1}_{n,m}$ associated with the measure $\frac{d\theta\, d\varphi}{4\pi^2\,Q(\theta,\varphi)}$ on $[-\pi,\pi]^2$ satisfy equations \eqref{2.17}.
\end{Theorem}
In this case when the equivalent conditions in \thref{th2.5} hold we have the following structural theorem.
\begin{Theorem}\label{th2.7}
Suppose that \eqref{2.16} holds, where $p(z,w)$ and $q(z,w)$ are stable 
polynomials of degrees $(n_1,m_1)$ and $(n_2,m_2)$, respectively. 
Let $\Phip_{k,l}(z,w)$ and $\Phiq_{k,l}(z,w)$ be the the (vector) polynomials orthogonal with respect to the measures 
$\frac{d\theta\,d\varphi}{4\pi^2|p(e^{i\theta},e^{i\varphi})|^2}$ and $\frac{d\theta\,d\varphi}{4\pi^2|q(e^{i\theta},e^{i\varphi})|^2}$, respectively. Then
\begin{equation}\label{2.18}
\Phip_{n_1,m_1}(z,w)=\left[\begin{matrix} \pr(z,w)\\[-6pt] \\ \Phip_{n_1,m_1-1}(z,w) \end{matrix}\right], \quad 
\Phiq_{n_2,m_2}(z,w)=\left[\begin{matrix} \qr(z,w)\\[-6pt]\\ \Phiq_{n_2,m_2-1}(z,w) \end{matrix}\right].
\end{equation}
Moreover, if we set $n=n_1+n_2$ and $m=m_1+m_2$, then there exist unitary 
matrices $U\in M^{m,m}$, $V\in M^{m+1,m+1}$ such that 
\begin{subequations}\label{2.19}
\begin{equation}\label{2.19a}
U^{\dagger}\Phi_{n,m-1}(z,w)=\left[\begin{matrix} z^{n_2}q(1/z,w) \Phip_{n_1,m_1-1}(z,w) \\[-6pt]\\
\pr(z,w) w^{m_2-1}\overline{\Phiq_{n_2,m_2-1}}(z,1/w)\end{matrix}\right],
\end{equation}
and
\begin{equation}\label{2.19b}
V^{\dagger}\Phi_{n,m}(z,w)=\left[\begin{matrix} \pr(z,w)z^{n_2}q(1/z,w)\\[-6pt] \\ z^{n_2}q(1/z,w) \Phip_{n_1,m_1-1}(z,w) \\[-6pt]\\
\pr(z,w) w^{m_2}\overline{\Phiq_{n_2,m_2-1}}(z,1/w)\end{matrix}\right].
\end{equation}
\end{subequations}
\end{Theorem}
Roughly speaking, the above theorem allows to decompose the space of orthogonal polynomials associated with the functional in \eqref{2.16} as a sum of the two extreme cases:
\begin{itemize}
\item the {\em stable case} when $q(z,w)=1$;
\item the {\em anti-stable} case when $p(z,w)=1$.
\end{itemize}
As a corollary of the the proof we obtain also the following characterizations 
of these situations.
\begin{Corollary}\label{co2.8} 
For a positive moment functional $\cL$ defined on the space $\Pinm$ the following statements hold.
\begin{itemize}
\item[{\rm{(i)}}] There exists a stable polynomial $p(z,w)$ of degree at most 
$(n,m)$ such that 
\begin{equation}\label{2.20}
\cL(z^kw^l)
=\frac{1}{4\pi^2}\int\limits_{[-\pi,\pi]^2}\frac{e^{ik\theta}e^{il\varphi}}{|p(e^{i\theta},e^{i\varphi})|^2}\,d\theta\,d\varphi
\end{equation}
if and only if $\cK_{n,m}=0$. Moreover, we can take $p(z,w)=\phir_{n,m}^{m}(z,w)$.
\item[{\rm{(ii)}}] There exists a stable polynomial $q(z,w)$ of degree at most 
$(n,m)$ such that 
\begin{equation}\label{2.21}
\cL(z^kw^l)
=\frac{1}{4\pi^2}\int\limits_{[-\pi,\pi]^2}\frac{e^{ik\theta}e^{il\varphi}}{|q(e^{-i\theta},e^{i\varphi})|^2}\,d\theta\,d\varphi
\end{equation}
if and only if $\cK^{1}_{n,m}=0$.
\end{itemize}
\end{Corollary}
As a consequence of the above corollary, we obtain a simple characterization of the functionals which are tensor products of functionals on the circle.
\begin{Corollary}\label{co2.9} 
Let $\cL$ be a positive moment functional on the space $\Pinm$. Then, there exist a positive functional $\cL_z$ defined on 
$\Span\{z^k:|k|\leq n\}$ and a positive functional $\cL_w$ defined on $\Span\{w^l:|l|\leq m\}$ such that 
$\cL(z^kw^l)=\cL_z(z^k)\cL_w(w^l)$ if and only if $\cK_{n,m}=\cK^{1}_{n,m}=0$. In this case, $\phir_{n,m}^{m}(z,w)=\alpha(z)\beta(w)$, where $\alpha(z)$ and $\beta(w)$ are stable polynomials of degrees at most $n$ and $m$, respectively and 
\begin{equation}\label{2.22}
\cL(z^kw^l)
=\frac{1}{4\pi^2}\int\limits_{[-\pi,\pi]^2}\frac{e^{ik\theta}e^{il\varphi}}{|\alpha(e^{i\theta})\beta(e^{i\varphi})|^2}\,d\theta\,d\varphi.
\end{equation}
\end{Corollary}
Finally, the above results can be used to completely characterize the measures on $\Tset^2$ for which the corresponding coefficients $\Eh_{k,l}$ and $\Eht_{k,l}$ vanish after a particular point.
\begin{Theorem}\label{th2.10}
Let $\mu$ be a positive Borel measure supported on the bi-circle. Then $\mu$ 
is absolutely continuous with respect to Lebesgue measure with 
\begin{equation}\label{2.23}
d\mu=\frac{d\theta\, d\varphi}{4\pi^2|p(e^{i\theta},e^{i\varphi})q(e^{-i\theta},e^{i\varphi})|^2},
\end{equation}
where $p(z,w)$ and $q(z,w)$ are stable polynomials of degrees 
$(n_1,m_1)$ and $(n_2,m_2)$, respectively, with $n_1+n_2\leq n$, 
$m_1+m_2\leq m$ if and only if 
\begin{equation}\label{2.24}
\Eh_{k,l}=0\text{ and }\Eht_{k,l}=0\text{ for all }k\geq n+1,\quad l\geq m+1.
\end{equation}
Moreover, in this case we have
\begin{equation}\label{2.25}
\Eh_{k,l}=0, \text{ for } k\geq n+1,\ l\geq m-1\text{ and }\Eht_{k,l}=0, 
\text{ for } k\geq n-1,\ l\geq m+1.
\end{equation}
\end{Theorem}

\section{Preliminary results} \label{se3}
\subsection{Connection between bivariate and matrix orthogonal polynomials}\label{ss3.1}
The vector polynomial $\Phi_{n,m}(z,w)$ defined in \eqref{2.6} can be written as 
\begin{subequations}\label{3.1}
\begin{equation}\label{3.1a}
\Phi_{n,m}(z,w)=\Phi^{m}_{n}(z)\left[\begin{matrix} w^{m}\\ w^{m-1}\\ \vdots \\ 1\end{matrix}\right],
\end{equation}
where $\Phi^{m}_{n}(z)$ is a unique $(m+1)\times(m+1)$ matrix polynomial of degree $n$ in $z$. Similarly, the vector polynomial $\Phit_{n,m}(z,w)$ defined in \eqref{2.8} can be written as 
\begin{equation}\label{3.1b}
\Phit_{n,m}(z,w)=\Phit^{n}_{m}(w)\left[\begin{matrix} z^{n}\\ z^{n-1}\\ \vdots \\ 1\end{matrix}\right],
\end{equation}
\end{subequations}
where $\Phit^{n}_{m}(w)$ is a unique $(n+1)\times(n+1)$ matrix polynomial of degree $m$ in $w$. 
The recurrence relation \eqref{2.11a} and its tilde-analog are equivalent to the recurrence relations for the matrix-valued polynomials 
$\{\Phi^{m}_{n}(z)\}_{n\geq 0}$ and $\{\Phit^{n}_{m}(w)\}_{m\geq 0}$.

We will also need the following Christoffel-Darboux formula, which is a tilde analog of formula (4.1a)-(4.1c) in \cite{GW2}
\begin{equation}\label{3.2}
\begin{split}
&\Phitr_{n,m}(z,w)\Phitr_{n,m}(z_1,w_1)^{\dagger}-\Phitr_{n-1,m}(z,w)\Phitr_{n-1,m}(z_1,w_1)^{\dagger} \\
&\quad -w\wb_1\left[\Phit_{n,m}(z,w)^{T} \overline{\Phit_{n,m}(z_1,w_1)}-\Phit_{n-1,m}(z,w)^{T} \overline{\Phit_{n-1,m}(z_1,w_1)}\right]\\
&\qquad =(1-w\wb_1)\Phi_{n,m}(z,w)^{T} \overline{\Phi_{n,m}(z_1,w_1)},
\end{split}
\end{equation}
and its corollary (see equation (4.2) in \cite{GW2})
\begin{equation}\label{3.3}
\begin{split}
&\Phi_{n,m}(z,w)^{T}\,\overline{\Phi_{n,m}(z_1,w_1)}- \Phi_{n,m-1}(z,w)^{T}\,\overline{\Phi_{n,m-1}(z_1,w_1)}\\
&\quad = \Phit_{n,m}(z,w)^{T}\,\overline{\Phit_{n,m}(z_1,w_1)}-\Phit_{n-1,m}(z,w)^{T}\,\overline{\Phit_{n-1,m}(z_1,w_1)}.
\end{split}
\end{equation}

\subsection{Relations among the coefficients}\label{ss3.2}

We list below different relations among the coefficients defined in 
\eqref{2.12} needed in the paper. 

The tilde analog of formula~(3.52) on page 811 in \cite{GW2} can be written as 
follows
\begin{equation}\label{3.4}
\begin{split}
\Gt_{k+1,l}^{1}\Gt_{k+1,l}^{\dagger}=&\Gt_{k,l}^{\dagger}\Gt_{k,l}^{1}
+\It_{k,l}\Eh_{k+1,l}(\It^{1}_{k,l})^{T}\\
&\quad +\cKt^{1}_{k+1,l}(\bar{A}_{k+1,l-1})^{-1}\Eh_{k+1,l-1}^{\dagger}A_{k+1,l-1}\cKt^{\dagger}_{k+1,l}.
\end{split}
\end{equation}

We also need formulas (3.1), (3.4) and (3.6) from \cite{GB}:
\begin{subequations}\label{3.5}
\begin{align}
&\Eh_{k+1,l-1}=\G_{k,l}\Eh_{k+1,l}(\G^1_{k,l})^{T}+\cK_{k,l}(\cK^{1}_{k,l})^{T},\label{3.5a}\\
&\G_{k,l}\Eh_{k+1,l}I^{1}_{k,l}=A_{k+1,l-1}\cK_{k+1,l}-\cK_{k,l}\Gt^{1}_{k,l},\label{3.5b}\\
&I_{k,l}^{\dagger}\Eh_{k+1,l}(\G^{1}_{k,l})^{T}
=(\cK^{1}_{k+1,l})^{T}A_{k+1,l-1}^{T}-\Gt^{\dagger}_{k,l}(\cK^{1}_{k,l})^{T}.\label{3.5c}
\end{align}
\end{subequations}
Recall that if $\Eh_{k,l}=0$ then $A_{k,l}=I_{l+1}$ is the identity $(l+1)\times(l+1)$ matrix. Using this fact and 
the above formulas, we see that the following lemma holds.
\begin{Lemma}\label{le3.1} 
If
\begin{equation}\label{3.6}
\Eh_{k+1,l}=0 \text{ and }\Eh_{k+1,l-1}=0,
\end{equation}
then
\begin{subequations}\label{3.7}
\begin{align}
\cK_{k,l}(\cK^{1}_{k,l})^{T}&=0,\label{3.7a}\\
\Gt_{k+1,l}^{1}\Gt_{k+1,l}^{\dagger}&=\Gt_{k,l}^{\dagger}\Gt_{k,l}^{1},\label{3.7b}\\
\cK_{k+1,l}&=\cK_{k,l}\Gt^{1}_{k,l},\label{3.7c}\\
(\cK^{1}_{k+1,l})^{T}&=\Gt^{\dagger}_{k,l}(\cK^{1}_{k,l})^{T}.\label{3.7d}
\end{align}
\end{subequations}
\end{Lemma}

\subsection{Stability criterion}\label{ss3.3}
Throughout the paper we will use several times the following fact: a polynomial $p(z,w)$ is stable (i.e. non-vanishing for $|z|\leq 1$ and $|w|\leq 1$) if and only if 
\begin{itemize}
\item $p(z,w)\neq 0$ for $|z|=1$ and $|w|\leq 1$, and 
\item $p(z,w)\neq 0$ for $|z|\leq1$ and $|w|= 1$.
\end{itemize}
The above criterion is a simple corollary from the well-known stability criteria for bivariate polynomials, see for instance \cite{DGK3}.

\section{One sided stable polynomials}\label{se4}
In this section we prove \thref{th2.3}.  
\subsection{Proof of the implication (i)$\Rightarrow$(ii) in \thref{th2.3}}
Assume first that the conditions in \thref{th2.3}(i) hold, i.e. the moment functional $\cL$ is defined on $\Pinm$ by 
\begin{align}
&\cL(z^kw^l)=\frac{1}{4\pi^2}\int\limits_{[-\pi,\pi]^2}\frac{e^{ik\theta}e^{il\varphi}}{|p(e^{i\theta},e^{i\varphi})|^2}\,d\theta\,d\varphi,\nn\\
&\quad\text{where $p(z,w)$ is of degree $(n,m)$ nonzero for $ |z|=1$ and $|w|\leq1$.}\label{4.1}
\end{align}
We can use \eqref{4.1} to extend the functional $\cL$ on the space of all Laurent 
polynomials $\Cset[z,z^{-1},w,w^{-1}]$. Thus we can define vector polynomials 
$\Phi_{k,l}(z,w)$ for all $k,l\in\Nset_0$.

For every fixed $z=e^{i\theta}\in\Tset$, we denote by $\cLz$ the corresponding positive moment functional on the 
space $\Cset[w,w^{-1}]$ given by 
\begin{equation}\label{4.2}
\cLz(w^l)=\frac{1}{2\pi}\int\limits_{-\pi}^{\pi}\frac{e^{il\varphi}}{|p(e^{i\theta},e^{i\varphi})|^2}\,d\varphi.
\end{equation}
Similarly, for a polynomial $\phi(z,w)$ of degree $(k,l)$ we can fix $z=e^{i\theta}$ on the unit circle and consider the corresponding polynomial $\phi(e^{i\theta},w)$ of degree $l$ in $w$ which depends on the parameter $\theta$. We will denote by 
$\revz{\phi}(e^{i\theta},w)$ the reverse polynomial of 
$\phi(e^{i\theta},w)$, i.e. we set
$$\revz{\phi}(e^{i\theta},w)=w^{l}\phib(e^{-i\theta},1/w).$$

\begin{Lemma}\label{le4.1}
Suppose that \eqref{4.1} holds. Then with respect to $\cLz$ we have
\begin{subequations}\label{4.3}
\begin{align}
& p(e^{i\theta},w)\perp \{w^{l}:l>0\},\label{4.3a}\\
&\revz{p}(e^{i\theta},w) \perp \{w^{l}:l<m\},\label{4.3b}
\end{align}
\end{subequations}
and $\left\|p(e^{i\theta},w)\right\|=\left\|\revz{p}(e^{i\theta},w)\right\|=1$.
\end{Lemma}

\begin{proof}
We have 
$$\cLz\left(w^{l}\,\overline{p(e^{i\theta},w)}\right)=\frac{1}{2\pi}\int\limits_{-\pi}^{\pi}\frac{e^{il\varphi}\overline{p(e^{i\theta},e^{i\varphi})}}{|p(e^{i\theta},e^{i\varphi})|^2}d\varphi=-\frac{i}{2\pi}\oint\limits_{\Tset}\frac{w^{l-1}}{p(e^{i\theta},w)}dw=0,$$
for $l>0$ by Cauchy's residue theorem, establishing \eqref{4.3a}. 
The second orthogonality follows by a similar computation. The assertion about the norms of $p(e^{i\theta},w)$ and $\revz{p}(e^{i\theta},w)$ is straightforward.
\end{proof}
We would like to construct now polynomials $\{\phiz{l}(w)\}_{l\geq 0}$ 
orthonormal with respect to $\cLz$. From \leref{le4.1} it follows that we can 
take 
\begin{equation}\label{4.4}
\phiz{l}(w)=w^{l-m}\revz{p}(e^{i\theta},w), \text{ for }l\geq m.
\end{equation}

Let us denote by $\Cz{l}$ the $(l+1)\times (l+1)$ Toeplitz matrix 
associated with $\cLz$, i.e. if we put $\cz{j}=\cLz(w^{-j})$ then 
\begin{equation*}
\Cz{l} = \left[
\begin{matrix}
\cz{0} & \cz{-1} & \cdots & \cz{-l}
\\
\cz{1} & \cz{0} & \cdots & \cz{-l+1}
\\
\vdots &  & \ddots & \vdots \\ 
\cz{l} & \cz{l-1} & \cdots & \cz{0}
\end{matrix}
\right].
\end{equation*}
Recall that we can use the coefficients of the orthonormal polynomial 
$\phiz{l}(w)$ to compute the inverse of $\Cz{l-1}$ via the 
Gohberg-Semencul formula \cite[Theorem 6.2, page 88]{GF}. Explicitly, if we set 
\begin{equation}\label{4.5}
\overleftarrow{\phiz{l}}(w)=\sum_{j=0}^{l}r^{l}_jw^j,
\end{equation}
then
\begin{align}
(\Cz{l-1})^{-1}&=
\left[\begin{matrix}r^{l}_0 && &\bigcirc\\
r^{l}_1 &\ddots\\
\vdots \\
r^{l}_{l-1} &\cdots & &r^{l}_0
\end{matrix}\right]
\left[\begin{matrix}\overline{r^{l}_0} & \overline{r^{l}_1} &\dots & \overline{r^{l}_{l-1}}\\
&\ddots \\
& \\
\bigcirc && & \overline{r^{l}_0}\end{matrix}\right]\nn\\
&\qquad -
\left[\begin{matrix}\overline{r^{l}_{l}} && &\bigcirc\\
\overline{r^{l}_{l-1}} &\ddots\\
\vdots \\
\overline{r^{l}_{1}} &\cdots & &\overline{r^{l}_l}
\end{matrix}\right]
\left[\begin{matrix}r^{l}_l & r^{l}_{l-1} &\dots & r^{l}_{1}\\
&\ddots \\
& \\
\bigcirc && & r^{l}_{l}\end{matrix}\right].\label{4.6}
\end{align}
\begin{Lemma}\label{le4.2}
Suppose that \eqref{4.1} holds for all $(k,l)\in\Zset^2$. Then
\begin{equation}\label{4.7}
\Eh_{k,l}=0 \text{ for }k\geq n+1 \text{ and }l\geq m-1.
\end{equation}
\end{Lemma}

\begin{proof}
Note that for fixed $l\geq m-1$, the matrix polynomials 
$\{\Phi_k^{l}(z)\}_{k\geq 0}$ defined in \ssref{ss3.1} are orthonormal on 
$[-\pi,\pi]$ with respect to the matrix weight $\frac{1}{2\pi}\Cz{l}$, i.e.
\begin{equation*}
\frac{1}{2\pi}\int_{-\pi}^{\pi}\Phi^{l}_{k}(e^{i\theta})\Cz{l}
[\Phi^{l}_{j}(e^{i\theta})]^{\dagger}d\theta=\delta_{kj}I_{l+1}.
\end{equation*}
From the theory of matrix-valued orthogonal polynomials it will follow that 
$\Eh_{k,l}=0$ for $k\geq n+1$ if we can show that $(\Cz{l})^{-1}$ is a 
(matrix) trigonometric polynomial in $\theta$ of degree at most $n$. This 
follows immediately from \eqref{4.4}, \eqref{4.5} and \eqref{4.6}.
\end{proof}

\begin{Lemma}\label{le4.3}
Suppose that equation  \eqref{4.7} holds. Then
\begin{equation}\label{4.8}
\cK_{k,l}\left[\Gt^{1}_{k,l}\Gt^{\dagger}_{k,l}\right]^j(\cK^{1}_{k,l})^{T}=0, 
\text{ for all }j\geq 0,\quad k\geq n,\quad l\geq m.
\end{equation}
\end{Lemma}

\begin{proof}
From \leref{le3.1} we see that equations \eqref{3.7} hold as long as 
$k\geq n$ and $l\geq m$.
First, we would like to show by induction on $j\in\Nset_0$ that
\begin{equation}\label{4.9}
\cK_{k+j,l}(\cK^{1}_{k+j,l})^{T}=\cK_{k,l}\left[\Gt^{1}_{k,l}\Gt^{\dagger}_{k,l}\right]^j(\cK^{1}_{k,l})^{T}, \text{ for }k\geq n,\quad l\geq m.
\end{equation}
If $j=0$, the above statement is obvious. Suppose now that \eqref{4.9} holds 
for some $j\geq 0$. From \eqref{3.7b} it follows that 
\begin{equation*}
\left[\Gt^{1}_{k,l}\Gt^{\dagger}_{k,l}\right]^{j+1}
=\Gt^{1}_{k,l}\left[\Gt^{1}_{k+1,l}\Gt^{\dagger}_{k+1,l}\right]^{j} \Gt^{\dagger}_{k,l}.
\end{equation*}
Using the above formula we find
\begin{align*}
&\cK_{k,l}\left[\Gt^{1}_{k,l}\Gt^{\dagger}_{k,l}\right]^{j+1}(\cK^{1}_{k,l})^{T} 
= \cK_{k,l}\Gt^{1}_{k,l}\left[\Gt^{1}_{k+1,l}\Gt^{\dagger}_{k+1,l}\right]^{j} \Gt^{\dagger}_{k,l}(\cK^{1}_{k,l})^{T}\\
&\qquad=\cK_{k+1,l}\left[\Gt^{1}_{k+1,l}\Gt^{\dagger}_{k+1,l}\right]^{j} (\cK^{1}_{k+1,l})^{T}\text{ (by equations \eqref{3.7c} and \eqref{3.7d})}\\
&\qquad=\cK_{k+1+j,l}(\cK^{1}_{k+1+j,l})^{T} \text{ (by the induction hypothesis)},
\end{align*}
establishing \eqref{4.9} for $j+1$ and completing the induction. From 
\eqref{3.7a} we see that the left-hand side of \eqref{4.9} is equal to $0$ leading to \eqref{4.8}.
\end{proof}

\begin{proof}[Proof of the implication (i)$\Rightarrow$(ii) in \thref{th2.3}]
The proof follows immediately from \leref{le4.2} and \leref{le4.3}.
\end{proof}

\subsection{Proof of the implication (ii)$\Rightarrow$(i) in \thref{th2.3}}

The key ingredient of the proof in the opposite direction, which also explains the construction of the polynomial $p(z,w)$, is the following lemma.

\begin{Lemma}\label{le4.4}
Let $\cL$ be a positive moment functional defined on $\Pinm$. Suppose that there exist unitary matrices $\Ut\in M^{n,n}$ and 
$\Vt\in M^{n+1,n+1}$ such that
\begin{subequations}\label{4.10}
\begin{align}
\Ut^{\dagger}\Phit_{n-1,m}(z,w)&=\left[\begin{matrix}\Psit_{n-1,m}^{(1)}(z,w)\\ \Psit_{n-1,m}^{(2)}(z,w)\end{matrix} \right],\label{4.10a} \\
\intertext{ and } 
\Vt^{\dagger}\Phit_{n,m}(z,w)&=\left[\begin{matrix}\psit_{n,m}^{n}(z,w)\\ z\Psit_{n-1,m}^{(1)}(z,w)\\ \Psit_{n-1,m}^{(2)}(z,w)\end{matrix} \right],\label{4.10b}
\end{align}
\end{subequations}
where $\Psit_{n-1,m}^{(j)}(z,w)$ is an $n_j$-dimensional vector whose components are polynomials of degrees at most $(n-1,m)$ with $n_1+n_2=n$, and $\psit_{n,m}^{n}(z,w)$ is a polynomial of degree at most $(n,m)$. Then 
\begin{equation}\label{4.11}
p(z,w)=\psitr_{n,m}^{n}(z,w)=z^{n}w^{m}\overline{\psit_{n,m}^{n}(1/\zb,1/\wb)}
\end{equation}
is a polynomial of degree at most $(n,m)$, nonzero for $|z|=1$, $|w|\leq 1$ and equations \eqref{2.13} and \eqref{2.15} hold.
\end{Lemma}

\begin{proof}
From equations \eqref{4.10} and \eqref{4.11} it follows that
\begin{subequations}\label{4.12}
\begin{align}
&\Phitr_{n,m}(z,w)\Phitr_{n,m}(z_1,w_1)^{\dagger}-\Phitr_{n-1,m}(z,w)\Phitr_{n-1,m}(z_1,w_1)^{\dagger}  \nn\\
&\quad =p(z,w)\overline{p(z_1,w_1)} \qquad\text{ for }z\zb_1=1,\label{4.12a}
\intertext{and}
&\Phit_{n,m}(z,w)^{T} \overline{\Phit_{n,m}(z_1,w_1)}
-\Phit_{n-1,m}(z,w)^{T} \overline{\Phit_{n-1,m}(z_1,w_1)} \nn\\
&\qquad = \pr(z,w)\overline{\pr(z_1,w_1)} \qquad\text{ for }z\zb_1=1,
\label{4.12b}
\end{align}
\end{subequations}
where $\pr(z,w)=z^{n}w^{m}\overline{p(1/\zb,1/\wb)}$.
Plugging equations \eqref{4.12} in \eqref{3.2} we obtain
\begin{align}
&p(z,w)\overline{p(1/\zb,w_1)} -w\wb_1\pr(z,w)\overline{\pr(1/\zb,w_1)}\nn\\
&\qquad =(1-w\wb_1)\Phi_{n,m}(z,w)^{T} \overline{\Phi_{n,m}(1/\zb,w_1)}.
\label{4.13}
\end{align}
Using the last equation we can prove that $p(z,w)$ is nonzero for $|z|=1$ and 
$|w|\leq 1$. Recall first that the vector polynomials $\Phi_{n,m}(z,w)$ can be 
connected to the matrix polynomials $\Phi^{m}_{n}(z)$ via \eqref{3.1a}. 
Moreover the matrix-valued orthogonal polynomials $\{\Phi_{k,m}(z)\}_{k=0}^{n}$ constructed in \ssref{ss3.1} are orthonormal with respect to the matrix inner product 
\begin{subequations}\label{4.14}
\begin{equation}\label{4.14a}
\langle A, B\rangle=\cL(AM_{m}(w)B^{\dagger}), 
\end{equation}
where $M_m(w)$ is the $(m+1)\times (m+1)$ Toeplitz matrix
\begin{align}
 M_m(w)&= 
 \left[\begin{matrix} w^{m}\\ w^{m-1}\\ \vdots \\ 1\end{matrix}\right]\, 
  \left[\begin{matrix} w^{-m} & w^{-m+1} & \dots & 1\end{matrix}\right]\nn\\
& \qquad=
 \left[
\begin{matrix}
1 & w & \dots & w^{m}
\\
w^{-1} & 1 & \dots & w^{m-1}
\\
\vdots &  & \ddots & \vdots \\ 
w^{-m} & w^{-m+1} & \dots & 1
\end{matrix}
\right].\label{4.14b}
\end{align}
\end{subequations}
In particular, from the theory of matrix-valued orthogonal polynomials we know that $\det[\Phi^{m}_{n}(z)]\neq0$ for $|z|\geq 1$. This implies that 
\begin{equation}\label{4.15}
\Phi_{n,m}(z,w) \text{ is a nonzero vector for }|z|=1 \text{ and } w\in\Cset.
\end{equation}
Suppose first that $p(z_0,w_0)=0$ for some $|z_0|=1$ and $|w_0|<1$. Then 
using \eqref{4.13} with $z=z_0$ and $w=w_1=w_0$ we obtain
\begin{equation*}
-|w_0|^2|\pr(z_0,w_0)|^2=(1-|w_0|^2)\Phi_{n,m}(z_0,w_0)^{T} \overline{\Phi_{n,m}(z_0,w_0)}.
\end{equation*}
Since the left-hand side of the above equation is $\leq 0$ and the right-hand side is $\geq 0$, we see that $\Phi_{n,m}(z_0,w_0)$ must be the zero vector, which contradicts \eqref{4.15}. 

Suppose now that $p(z_0,w_0)=0$ for some $|z_0|=1$ and $|w_0|=1$. Then $\pr(z_0,w_0)=0$ and therefore equation \eqref{4.13} with $z=z_0$, $w=w_0$ and $w_1\neq w_0$ gives
\begin{equation*}
\Phi_{n,m}(z_0,w_0)^{T} \overline{\Phi_{n,m}(z_0,w_1)}=0 \text{ for all }
w_1\neq w_0,
\end{equation*}
which implies that $\Phi_{n,m}(z_0,w_0)$ is the zero vector leading to a 
contradiction, thus proving the required stability for $p(z,w)$.

Note that equation \eqref{2.15} follows easily from \eqref{4.12a} and \eqref{3.3}. Thus,
it remains to prove that equation \eqref{2.13} holds. Let us denote by 
$p_l(z)$ the coefficient of $w^l$ in $p(z,w)$, i.e. we set
\begin{equation}
p(z,w)=\sum_{l=0}^{m}p_l(z)w^{l}.
\end{equation}
Then a straightforward computation shows that for $|z|=1$ we have
\begin{equation}\label{4.17}
\begin{split}
&\frac{p(z,w)\overline{p(z,w_1)} -w\wb_1\pr(z,w)\overline{\pr(z,w_1)}}{1-w\wb_1}\\
&\quad=\left[\begin{matrix} 1 & w & \cdots &w^{m}\end{matrix}\right]
\left(\rule{0cm}{1.4cm}\right.
\left[\begin{matrix}p_0(z) && &\bigcirc\\
p_1(z) &\ddots\\
\vdots \\
p_{m}(z) &\cdots & &p_0(z)
\end{matrix}\right]
\left[\begin{matrix}\overline{p_0(z)} & \overline{p_1(z)} &\dots & 
\overline{p_m(z)}\\
&\ddots \\
& \\
\bigcirc && & \overline{p_0(z)}\end{matrix}\right]\\
&\qquad -
\left[\begin{matrix}0 && & &\bigcirc\\
\overline{p_m(z)} &\ddots\\
\vdots \\
\overline{p_1(z)} &\cdots & & \overline{p_m(z)}&  0
\end{matrix}\right]
\left[\begin{matrix}0 & p_m(z) &\dots & p_1(z)\\
&\ddots & &\vdots\\
& & &p_m(z)\\
\bigcirc && & 0\end{matrix}\right]
\left. \rule{0cm}{1.4cm}\right)
\left[\begin{matrix} 1 \\ \wb_1 \\ \vdots \\ \wb_1^{m}\end{matrix}\right].
\end{split}
\end{equation}
From \eqref{3.1a} we see that 
\begin{equation}\label{4.18}
\Phi_{n,m}(z,w)^{T} \overline{\Phi_{n,m}(z,w_1)}=
\left[\begin{matrix} 1 & w & \cdots &w^{m}\end{matrix}\right] J_m\Phi^{m}_{n}(z)^{T} \overline{\Phi^{m}_{n}(z)}J_m
\left[\begin{matrix} 1 \\ \wb_1 \\ \vdots \\ \wb_1^{m}\end{matrix}\right],
\end{equation}
where $J_m=[\delta_{i,m-j}]_{0\leq i,j\leq m}$. From equations \eqref{4.13}, \eqref{4.17} and \eqref{4.18} it follows that for 
$|z|=1$ we have 
\begin{equation}\label{4.19}
\begin{split}
&\left[\begin{matrix}p_0(z) && &\bigcirc\\
p_1(z) &\ddots\\
\vdots \\
p_{m}(z) &\cdots & &p_0(z)
\end{matrix}\right]
\left[\begin{matrix}\overline{p_0(z)} & \overline{p_1(z)} &\dots & 
\overline{p_m(z)}\\
&\ddots \\
& \\
\bigcirc && & \overline{p_0(z)}\end{matrix}\right]\\
&\qquad -
\left[\begin{matrix}0 && & &\bigcirc\\
\overline{p_m(z)} &\ddots\\
\vdots \\
\overline{p_1(z)} &\cdots & & \overline{p_m(z)}&  0
\end{matrix}\right]
\left[\begin{matrix}0 & p_m(z) &\dots & p_1(z)\\
&\ddots & &\vdots\\
& & &p_m(z)\\
\bigcirc && & 0\end{matrix}\right]\\
&\qquad =J_m\Phi^{m}_{n}(z)^{T} \overline{\Phi^{m}_{n}(z)}J_m.
\end{split}
\end{equation}
Since $p(z,w)$ is nonzero for $|z|=1$ and $|w|\leq 1$ we see that for fixed $z=e^{i\theta}$ on the unit circle, 
$\phiz{m}(w)=w^{m}\overline{p(e^{i\theta},1/\wb)}$ is an orthonormal polynomial of degree $m$ with respect to the (parametric) moment functional $\cLz$, with moments 
\begin{equation}\label{4.20}
\cz{l}=\cLz(w^{-l})=\frac{1}{2\pi}\int\limits_{-\pi}^{\pi}\frac{e^{-il\varphi}}{|p(e^{i\theta},e^{i\varphi})|^2}\,d\varphi, \text{ for }|l|\leq m.
\end{equation}
From Gohberg-Semencul formula (see \cite[Theorem 6.1, page 86]{GF}) it follows that the left-hand side of equation \eqref{4.19} is the inverse of the Toeplitz matrix 
\begin{equation*}
\Cz{m} = \left[
\begin{matrix}
\cz{0} & \cz{-1} & \cdots & \cz{-m}
\\
\cz{1} & \cz{0} & \cdots & \cz{-m+1}
\\
\vdots &  & \ddots & \vdots \\ 
\cz{m} & \cz{m-1} & \cdots & \cz{0}
\end{matrix}
\right].
\end{equation*}
Since $J_m^2=I_{m+1}$ and $J_m\Cz{m} J_m=(\Cz{m})^{T}$, equation \eqref{4.19} gives
 \begin{equation}\label{4.21}
\Cz{m}=[\Phi^{m}_{n}(z)^{\dagger}\,\Phi^{m}_{n}(z)]^{-1}, \text{ where }z=e^{i\theta}.
\end{equation}
From the theory of matrix-valued orthogonal polynomials we know that the matrix 
weight on the right-hand side of \eqref{4.21} generates the same matrix-valued orthonormal polynomials $\{\Phi^{m}_{k}(z)\}_{0\leq k\leq n}$ and therefore
\begin{equation*}
\cL(z^{k}M_m(w))=\frac{1}{2\pi}\int\limits_{-\pi}^{\pi}e^{ik\theta}\Cz{m}d\theta\quad  \text{ for }-n\leq k\leq n.
\end{equation*}
From the first row of the last matrix equation we find that for $-n\leq k\leq n$ we have
\begin{equation}\label{4.22}
\cL(z^{k}w^{l})=\frac{1}{2\pi}\int\limits_{-\pi}^{\pi}e^{ik\theta}\cz{-l}d\theta.
\end{equation}
The proof of \eqref{2.13} follows at once from equations \eqref{4.20} and \eqref{4.22}.
\end{proof}

\begin{Remark}\label{re4.5}
To complete the proof of \thref{th2.3} we need to show that equation \eqref{2.14} implies the existence of unitary matrices $\Ut$ and $\Vt$ such that equations \eqref{4.10} hold. For moment functionals satisfying \eqref{2.16} the existence of such matrices follows easily from the tilde analog of formulas \eqref{2.19}. 
In general (for one-sided stability) the construction of $\Ut$ and $\Vt$ is the content of the next two lemmas.

Note also that if we know $\Vt$ we can compute explicitly $p(z,w)$ in equation \eqref{2.13} from \eqref{4.10b} and \eqref{4.11}. In \thref{th2.3} we gave the simplest formula for $|p(z,w)|^2$, which involves only the orthogonal polynomials. However, one can easily extract from the proof of \leref{le4.4} other formulas which can be used in practice to compute $p(z,w)$. For instance, we can use \eqref{4.12a} (which is stronger than \eqref{2.15}), or setting $w_1=0$ in \eqref{4.13} we obtain
$$p(z,w)\pb(1/z,0)=\Phi_{n,m}(z,w)^{T} \overline{\Phi_{n,m}}(1/z,0),$$
which gives $p(z,w)$ up to a factor depending only on $z$.
\end{Remark}

\begin{Lemma} \label{le4.6}
Let $\cK$ and $\cK^{1}$ be $m\times n$ matrices, and let
$r= \rank(\cK)$, $r^{1}=\rank(\cK^{1})$. Then the following
conditions are equivalent:
\begin{itemize}
\item[(i)] $\cK(\cK^{1})^{T}=0$;
\item[(ii)] We have 
\begin{equation}\label{4.23}
\cK=US\Ut^{\dagger}, \qquad \cK^{1}=U^{1}S^{1}\Ut^{T}, 
\end{equation}
where $U, U^{1}\in M^{m,m}$, $\Ut\in M^{n,n}$ are unitary and $S,S^{1}$ are 
$m\times n$ ``diagonal'' matrices with block structures of the form
\begin{subequations}\label{4.24}
\begin{equation}\label{4.24a}
S=\left[\begin{array}{ccc|c} 
s_1 & & & \\
& \ddots & & 0 \\
& & s_{r} & \\
\hline
 & 0& & 0
\end{array}\right],
\end{equation}
and
\begin{equation}\label{4.24b}
S^{1}=\left[\begin{array}{c|ccc} 
0 & & 0 & \\
\hline
  & s^{1}_1 & &   \\
0  &      & \ddots & \\
   &  & & s^{1}_{r^1}
\end{array}\right],
\end{equation}
\end{subequations}
with positive $s_{1},\dots s_{r}$, $s^{1}_{1},\dots,s^{1}_{r^1}$ and 
$r+r^{1}\leq n$.
\end{itemize}
\end{Lemma}

\begin{Remark*}
Note the the condition $r+r^1\le n$ implies $S(S^1)^T=0$.
\end{Remark*}

\begin{proof}
We focus on the implication (i) $\Rightarrow$ (ii), since the other direction 
is obvious. Consider $A=\cK^{\dagger}\cK$ and $B=(\cK^{1})^{T}\,\overline{\cK^{1}}$. 
Note that $A$ and $B$ are hermitian $n\times n$ matrices such that $AB=BA=0$.
Hence, there exists an orthonormal basis $(\ut_1,\ut_2,\dots, \ut_n)$
for $\Cset^n$ which diagonalizes $A$ and $B$, i.e.
\begin{equation}\label{4.25}
\cK^{\dagger}\cK \ut_j=\lambda_j \ut_j,\quad (\cK^{1})^{T}\overline{\cK^{1}} \ut_j=\mu_j\ut_j
\text{ and } 
\ut_i^{\dagger} \ut_j=\delta_{ij}.
\end{equation}
Let $\Ut$ be the unitary matrix with columns $\ut_1,\dots, \ut_n$.
From \eqref{4.25} we see that
$$\lambda_j=\|\cK \ut_j\|^2\ge 0\text{ and } \mu_j=\|\overline{\cK^{1}} \ut_j\|^2\ge 0.$$
Moreover, the fact $\cK(\cK^{1})^{T}=0$ implies that $\lambda_j\mu_j=0$ for all
$j$.
Suppose now that
\begin{itemize}
\item $\lambda_1,\lambda_2,\dots\lambda_r$ are positive and 
$\lambda_{r+1}=\cdots = \lambda_n=0$;
\item $\mu_{n-r^{1}+1},\dots, \mu_n$ are positive and 
$\mu_1=\mu_2=\cdots=\mu_{n-r^{1}}=0$.
\end{itemize}
Set
\begin{align*}
&s_i =\sqrt{\lambda_i}&& \text{for }i=1,2,\dots,r\\
&s^{1}_j =\sqrt{\mu_{n-r^{1}+j}} &&\text{for }j=1,2,\dots,r^{1}.
\end{align*}
Consider the sets of vectors
\begin{subequations}\label{4.26}
\begin{align}
T&=\left\{ u_i=\frac{\cK \ut_i}{s_i}: i=1,2,\dots, r\right\} \label{4.26a}\\
\intertext{and} 
T_1&=\left\{ u_j=\frac{\cK^{1} \bar{\ut}_{j+n-m}}{s^{1}_{j-m+r^1}}:
j=m-r^{1}+1,\dots, m\right\}.\label{4.26b}
\end{align}
\end{subequations}
Using \eqref{4.25} it is easy to see that $T$ and $T_1$ are orthonormal sets of vectors. Extending the set $T$ to an orthonormal basis for $\Cset^m$ and constructing a matrix with columns these vectors we obtain a unitary $m\times m$ matrix $U$; extending the set $T_1$ to an orthonormal basis for $\Cset^m$ and constructing a matrix with columns these vectors we obtain a unitary $m\times m$ matrix $U^1$. With these matrices one can check that \eqref{4.23} holds.
\end{proof}

\begin{Remark}\label{re4.7}
If know that $\cK(\cK^1)^T=0$ and $\cK^{\dagger}\cK^1=0$ (which are satisfied by the matrices $\cK=\cK_{n,m}$ and $\cK^1=\cK^1_{n,m}$ in \thref{th2.5}) then we can choose $U=U^{1}$ in equation \eqref{4.23}. Indeed, using the notations in the proof of \leref{le4.6} we see that vectors in $T$ are perpendicular to the vectors in $T_1$. Extending the orthonormal set $T\cup T_1$ to an orthonormal basis for $\Cset^m$ we can construct a unitary matrix $U=U^1$ with columns these vectors.
\end{Remark}

\begin{Lemma} \label{le4.8}
Let $G$ be an $n\times n$ matrix. Suppose that 
\begin{equation}\label{4.27}
\begin{split}
(G^k)_{i,j}=0 \text{ for all }&i=1,2,\dots,r, \quad j=n-r^{1}+1,n-r^{1}+2,\dots,n,\\ 
\text{ and }& k=1,2,\dots,n-1,
\end{split}
\end{equation}
where $r+r^{1}<n$. Then, there exists a unitary $n\times n$ block matrix $\cEt$ of 
the form
\begin{equation}\label{4.28}
\cEt=\left[
\begin{array}{c|c|c} 
I_{r}&0&0\\
\hline
0& *&0\\
\hline
0& 0& I_{r^{1}}
\end{array}\right],
\end{equation}
such that the matrix $\cEt^{\dagger}G\cEt$ has the following block structure
\begin{equation}\label{4.29}
\cEt^{\dagger}G \cEt=\left[
\begin{array}{c|c} 
*&0\\
\hline
*& *
\end{array}\right],
\end{equation}
where the zero block in equation \eqref{4.29} above is an 
$n_1\times n_2$ matrix with $n_1\geq r$, $n_2\geq r^{1}$ and $n_1+n_2=n$.
\end{Lemma}

\begin{proof}
Let $\{e_1,e_2,\dots,e_n\}$ be the standard basis for $\Cset^n$, and let 
$$W_0=\Span\{e_{n-r^{1}+1},e_{n-r^{1}+2},\dots,e_{n}\}.$$
Consider the space 
\begin{equation}\label{4.30}
W=W_0+GW_0+\cdots+G^{n-1}W_0.
\end{equation}
By Cayley-Hamilton theorem, $W$ is the minimal subspace of $\Cset^n$ which 
is $G$-invariant and contains $W_0$. 
From equation \eqref{4.27} it follows that 
\begin{equation}\label{4.31}
\Span\{e_1,e_2,\dots,e_r\}\subset W^{\perp}.
\end{equation}
Let 
\begin{itemize}
\item $\{v_1,v_2,\dots,v_{n_1}\}$ be an orthonormal basis for $W^{\perp}$ which extends $\{e_1,e_2,\dots,e_r\}$, i.e. $v_j=e_j$ for 
$j=1,2,\dots,r$;
\item $\{v_{n_1+1},v_{n_1+2},\dots,v_{n}\}$ be an orthonormal basis for $W$ which extends $\{e_{n-r^{1}+1},e_{n-r^{1}+2},\dots,e_{n}\}$, i.e. $v_j=e_j$ for $j>n-r^{1}$.
\end{itemize}
Then the unitary matrix $\cEt$ with columns $v_1,v_2,\dots,v_{n}$ will have the block structure given in \eqref{4.28}, and the $G$-invariance of $W$ implies equation \eqref{4.29}.
\end{proof}

\begin{proof}[Proof of the implication (ii)$\Rightarrow$(i) in \thref{th2.3}]
Applying \leref{le4.6} with 
$$\cK=\cK_{n,m}\quad\text{ and }\quad \cK^{1}=\cK^{1}_{n,m}$$
we see that there exist unitary matrices $U,U^{1}\in M^{m,m}$, $\Ut\in M^{n,n}$ 
such that equations \eqref{4.23}, \eqref{4.24} hold. Moreover, applying 
\leref{le4.8} with 
$$G=\Ut^{\dagger}\Gt^{1}_{n,m}\Gt^{\dagger}_{n,m} \Ut,$$
we see that $\Ut$ can be modified (if necessary), so that equations 
\eqref{4.23}, \eqref{4.24} hold and 
\begin{equation}\label{4.32}
G=\Ut^{\dagger}\Gt^{1}_{n,m}\Gt^{\dagger}_{n,m} \Ut=\left[
\begin{array}{c|c} 
*&0\\
\hline
*& *
\end{array}\right],
\end{equation}
where the zero block in the equation above is an 
$n_1\times n_2$ matrix with $n_1\geq r=\rank(\cK_{n,m})$, 
$n_2\geq r^{1}=\rank(\cK^{1}_{n,m})$ and $n_1+n_2=n$.

Replacing $\cK_{n,m}$ and $\cK^{1}_{n,m}$ in the tilde analogs of equations \eqref{2.11c} and \eqref{2.11d} with the expressions given in \eqref{4.23} we find
\begin{subequations}\label{4.33}
\end{subequations}
$$\Ut^{\dagger}\Gt_{n,m} \Phit_{n,m} = \Ut^{\dagger}\Phit_{n-1,m} - S^{T} U^{\dagger}\Phi_{n,m-1}, \eqno{(\text{\ref{4.33}c})}$$
and
$$\Ut^{\dagger}\Gt_{n,m}^1 \Phit_{n,m} = z \Ut^{\dagger}\Phit_{n-1,m} - (S^1)^{T}(U^1)^{T}\Phir_{n,m-1}^T. \eqno{(\text{\ref{4.33}d})}$$
With $n_1$ and $n_2$ fixed above, we will use the following notation: for an $n$-dimensional 
vector $\Psit$, we denote by $\Psit^{(1)}$ (resp. $\Psit^{(2)}$) the vector 
which consists of the first $n_1$ (resp. the last $n_2$) entries of the 
vector $\Psit$. Thus if we set
\begin{equation}\label{4.34}
\Psit_{n-1,m}=\Ut^{\dagger}\Phit_{n-1,m},
\end{equation}
then the vector $\Psit_{n-1,m}$ can be represented in the block form
\begin{equation}\label{4.35}
\Psit_{n-1,m}=\left[\begin{matrix}\Psit_{n-1,m}^{(1)}\\ \Psit_{n-1,m}^{(2)}\end{matrix}\right].
\end{equation}
With this choice of a unitary matrix $\Ut$, we want to show that there exists 
a unitary matrix $\Vt$ such that 
\eqref{4.10b} holds. Since the bottom $n_2$ rows of the matrix 
$S^T$ are equal to $0$, we see from equations (\ref{4.33}c) and \eqref{4.35} 
that
\begin{subequations}\label{4.36}
\end{subequations}
$$(\Ut^{\dagger}\Gt_{n,m} \Phit_{n,m})^{(2)} = \Psit_{n-1,m}^{(2)}. \eqno{(\text{\ref{4.36}c})}$$
Similarly, since the first $n_1$ rows of the matrix $(S^{1})^T$ are equal to $0$, we see from equations (\ref{4.33}d) and \eqref{4.35} 
that
$$(\Ut^{\dagger}\Gt_{n,m}^1 \Phit_{n,m})^{(1)} = z\Psit_{n-1,m}^{(1)}. \eqno{(\text{\ref{4.36}d})}$$
Equations \eqref{4.36} show that the entries of the vector polynomials 
$z\Psit_{n-1,m}^{(1)}$ and $\Psit_{n-1,m}^{(2)}$ are linear combinations of the 
entries of the vector polynomial $\Phit_{n,m}$ and therefore, they are orthogonal 
with respect to $\cL$ to all monomials of degree at most $(n,m-1)$.
Moreover, since $\Ut$ is unitary, it follows from \eqref{4.34} and 
\eqref{4.35} that the entries of each of the vectors $z\Psit_{n-1,m}^{(1)}$ and 
$\Psit_{n-1,m}^{(2)}$ form orthonormal sets of polynomials of degrees at most 
$(n,m)$ with respect to $\cL$. Finally, from equations \eqref{4.32} and 
\eqref{4.36} we see that the entries of the vector $z\Psit_{n-1,m}^{(1)}$ are 
perpendicular to the entries of the vector $\Psit_{n-1,m}^{(2)}$. Therefore, 
all the entries in the vectors $z\Psit_{n-1,m}^{(1)}$ and $\Psit_{n-1,m}^{(2)}$ 
form an orthonormal set of $n$ polynomials, which can be extended by adding a 
polynomial $\psit_{n,m}^{n}(z,w)$ to an orthonormal set of polynomials of degree 
at most $(n,m)$, perpendicular to all polynomials of degree at most $(n,m-1)$. The transition matrix between this set and the orthonormal polynomials 
$\{\phit_{n,m}^{s}(z,w)\}_{s=0,1,\dots,n}$ is a unitary matrix whose transpose is a unitary matrix $\Vt$ satisfying 
equation \eqref{4.10b}. The proof now follows from \leref{le4.4}.
\end{proof}

\section{Proofs of the theorems in the splitting case}\label{se5}
\subsection{Proof of \thref{th2.5}}\label{ss5.1}

The proof of implication (i)$\Rightarrow$(ii) follows easily from \thref{th2.3}
and its tilde analog. Indeed, note that if \eqref{2.16} holds, then we have 
\begin{subequations}\label{5.1}
\begin{align}
\cL(z^kw^l)
&=\frac{1}{4\pi^2}\int\limits_{[-\pi,\pi]^2}\frac{e^{ik\theta}e^{il\varphi}}
{|P(e^{i\theta},e^{i\varphi})|^2}\,d\theta\,d\varphi\label{5.1a}\\
&=\frac{1}{4\pi^2}\int\limits_{[-\pi,\pi]^2}\frac{e^{ik\theta}e^{il\varphi}}
{|Q(e^{i\theta},e^{i\varphi})|^2}\,d\theta\,d\varphi\label{5.1b}
\end{align}
\end{subequations}
where 
\begin{subequations}\label{5.2}
\begin{equation}\label{5.2a}
\begin{split}
P(z,w)&=p(z,w)z^{n_2}q(1/z,w) \text{ is a polynomial of degree at}\\ 
&\text{most $(n,m)$, stable for $|z|=1$ and $|w|\leq 1$,}
\end{split}
\end{equation}
and
\begin{equation}\label{5.2b}
\begin{split}
Q(z,w)&=p(z,w)w^{m_2}\bar{q}(z,1/w) \text{ is a polynomial of degree at}\\ 
&\text{most $(n,m)$, stable for $|z|\leq 1$ and $|w|= 1$.}
\end{split}
\end{equation}
Therefore, equation \eqref{2.17a} follows from 
\thref{th2.3} (i)$\Rightarrow$(ii) for the polynomial $P(z,w)$ and 
equation \eqref{2.17b} follows from the tilde analog of \thref{th2.3} 
(i)$\Rightarrow$(ii) for the polynomial $Q(z,w)$.
\end{subequations}
Conversely, suppose that equations \eqref{2.17} hold. Then, by 
\thref{th2.3} (ii)$\Rightarrow$(i) and its tilde analog, we deduce that there 
exist polynomials $P(z,w)$ and $Q(z,w)$ of degrees at most $(n,m)$ such that equations \eqref{5.1} hold and 
\begin{itemize}
\item[{(a)}] $P(z,w)$ is stable for $|z|=1$ and $|w|\leq 1$
\item[{(b)}] $Q(z,w)$ is stable for $|z|\leq1$ and $|w|= 1$.
\end{itemize}
Without any restrictions, we can assume that $P(z,w)$ and $Q(z,w)$ are not 
divisible by $z$ and $w$. From equation \eqref{2.15} for $P(z,w)$ and $Q(z,w)$ we see that $|P(z,w)|^2=|Q(z,w)|^2$ for all 
$(z,w)\in\Tset^2$ which implies that in the ring of Laurent polynomials 
$\Cset[z,z^{-1},w,w^{-1}]$ we have
\begin{equation}\label{5.3}
P(z,w)\bar{P}(1/z,1/w)=Q(z,w)\bar{Q}(1/z,1/w).
\end{equation}
Suppose now that we factor $P(z,w)$ into a product of irreducible factors $p_j(z,w)$ in $\Cset[z,w]$. Likewise we can factor $Q(z,w)$ into a product of 
irreducible factors $q_l(z,w)$ in $\Cset[z,w]$. 
Using \eqref{5.3} we see that we can factor 
$P(z,w)\bar{P}(1/z,1/w)=Q(z,w)\bar{Q}(1/z,1/w)$ in $\Cset[z,z^{-1},w,w^{-1}]$ in 
two ways as a product of irreducible factors
\begin{equation}\label{5.4}
\begin{split}
&P(z,w)\bar{P}(1/z,1/w)
=\prod_{j}p_j(z,w)\pb_j(1/z,1/w)\\
&\qquad\qquad=\prod_{l}q_l(z,w)\qb_l(1/z,1/w)=Q(z,w)\bar{Q}(1/z,1/w).
\end{split}
\end{equation}
Since $\Cset[z,z^{-1},w,w^{-1}]$ is a unique factorization domain, it follows 
that for every $j$, there exists a unique $l$ such that exactly one of the 
following holds:
\begin{itemize}
\item[{(I)}]  $p_j(z,w)$ and $q_l(z,w)$ are associates in 
$\Cset[z,z^{-1},w,w^{-1}]$;
\item[{(II)}] $p_j(z,w)$ and $\qb_l(1/z,1/w)$ are associates in 
$\Cset[z,z^{-1},w,w^{-1}]$.
\end{itemize}
Note that the units in $\Cset[z,z^{-1},w,w^{-1}]$ are of the form $cz^sw^r$, 
where $c\neq 0$, $s,r\in\Zset$.  Thus, we see that if (I) holds then with the 
normalization chosen above ($P$ and $Q$ are not divisible by $z$ and $w$) we 
must have $p_j(z,w)=c q_l(z,w)$ where $c$ is a nonzero constant. From 
properties (a) and (b) of the polynomials $P(z,w)$ and $Q(z,w)$ we deduce 
that $p_j(z,w)=c q_l(z,w)\neq 0$ when $|z|=1$, $|w|\leq 1$, and 
likewise $p_j(z,w)=c q_l(z,w)\neq 0$ when $|z|\leq 1$, $|w|= 1$. 
This shows that if (I) holds, then 
$p_j(z,w)\neq 0$ when $|z|\leq 1$ and $|w|\leq 1$.

If (II) holds then $p_j(z,w)=cz^{s_j}w^{r_j}\qb_l(1/z,1/w)$ 
where $c\neq 0$, and $s_j,r_j$ are the minimal nonnegative integers for which 
$z^{s_j}w^{r_j}\qb_l(1/z,1/w)$ belongs to $\Cset[z,w]$. It is easy to see that 
$p_j(z,w)$ and $q_l(z,w)$ have the same degree $(s_j,r_j)$. From property (a) 
of the polynomial $P(z,w)$ we deduce that 
$z^{s_j}p_j(1/z,w)=cw^{r_j}\qb_l(z,1/w)\neq 0$ when 
$|z|=1$, $|w|\leq 1$. From property (b) of the polynomial $Q(z,w)$ we 
conclude that $z^{s_j}p_j(1/z,w)=cw^{r_j}\qb_l(z,1/w)\neq 0$ 
when $|z|\leq 1$, $|w|=1$. Thus we see that if (II) holds, then the 
polynomial $z^{s_j}p_j(1/z,w)$ has no zeros when $|z|\leq 1$ and $|w|\leq 1$.

Let $J_1$ (resp. $J_2$) denote the set of indices $j$ for which (I) 
(resp. (II)) holds. Then the polynomials $p(z,w)=\prod_{j\in J_1}p_j(z,w)$ and 
$q(z,w)=\prod_{j\in J_2}z^{s_j}p_j(1/z,w)$ satisfy the conditions in 
\thref{th2.5}(i), completing the proof.
\qed

\subsection{Proof of \thref{th2.7}}\label{ss5.2}
For the proof of \thref{th2.7} we summarize first some basic properties of the vector orthogonal polynomials $\Phi_{k,l}(z,w)$ associated with 
a moment functional $\cL$ of the form
\begin{align}
&\cL(z^kw^l)=\frac{1}{4\pi^2}\int\limits_{[-\pi,\pi]^2}\frac{e^{ik\theta}e^{il\varphi}}{|p(e^{i\theta},e^{i\varphi})|^2}\,d\theta\,d\varphi,\nn\\
&\quad\text{where $p(z,w)$ is of degree $(n,m)$ nonzero for $ |z|\leq1$ and $|w|\leq1$.}\label{5.5}
\end{align}
We define as usual $\pr(z,w)=z^{n}w^{m}\overline{p(1/\zb,1/\wb)}$. 
\begin{Lemma}\label{le5.1}
If \eqref{5.5} holds then 
\begin{subequations}\label{5.6}
\begin{align}
\cL(p(z,w)z^{-k}w^{-l})&=0 && \text{for all }\quad k\in\Zset, \quad l>0,\label{5.6.a}\\
\cL(\pr(z,w)z^{-k}w^{-l})&=0 && \text{for all }\quad k<n, \quad l\in\Zset.\label{5.6.b}
\end{align}
\end{subequations}
\end{Lemma}
\begin{proof}
The proof of \eqref{5.6.a} follows immediately by computing first the $w$ integral and by using equation \eqref{4.3a} in \leref{le4.1}. The proof of \eqref{5.6.b} follows by a similar computation, by evaluating first the $z$ integral.
\end{proof}

\begin{Lemma}\label{le5.2}
Suppose that \eqref{5.5} holds. Then the vector polynomial $\Phi_{n,m}(z,w)$ has the following block structure
\begin{equation}\label{5.7}
\Phi_{n,m}(z,w)=\left[\begin{matrix} \pr(z,w)\\ \Phi_{n,m-1}(z,w) \end{matrix}\right].
\end{equation}
Moreover, 
\begin{equation}\label{5.8}
\cL(\Phi_{n,m-1}(z,w)z^{-k}w^{-l})=0 \qquad \text{for all }k<n, \quad l\geq0.
\end{equation}
\end{Lemma}
\begin{proof}
Equation \eqref{5.7} follows from Theorem~7.2 in \cite{GW2}.  Plugging \eqref{5.7} and its tilde analog in \eqref{3.2} we obtain the following identity
\begin{equation*}
\begin{split}
&p(z,w)\overline{p(z_1,w_1)}-\pr(z,w)\overline{\pr(z_1,w_1)}\\
&\qquad=(1-w\wb_1)\Phi_{n,m-1}(z,w)^{T}\,\overline{\Phi_{n,m-1}(z_1,w_1)}\\
&\qquad\quad+(1-z\zb_1)\Phitr_{n-1,m}(z,w)\,\Phitr_{n-1,m}(z_1,w_1)^{\dagger}.
\end{split}
\end{equation*}
Thus, if take $z=z_1$ on the unit circle the last term above will vanish and we can rewrite the equation as follows
\begin{equation*}
\frac{p(z,w)\pb(1/z,\wb_1)-\pr(z,w) z^{-n}\wb_1^{m}p(z,1/\wb_1)}{1-w\wb_1}=\Phi_{n,m-1}(z,w)^{T}\,\overline{\Phi_{n,m-1}(z,w_1)}.
\end{equation*}
Using the matrix-valued polynomial $\Phi^{m-1}_{n}(z)$ defined in \eqref{3.1a} and its reverse $\Phir^{m-1}_{n}(z)=z^{n}\overline{\Phi^{m-1}_{n}(1/\zb)^{T}}$ we can replace in the last equation $\overline{\Phi_{n,m-1}(z,w_1)}$ by 
$$z^{-n}\Phir^{m-1}_{n}(z)^{T}\left[\begin{matrix} \wb_1^{m-1} & \wb_1^{m-2} & \cdots &1\end{matrix}\right]^{T},$$
and therefore we obtain
\begin{equation}\label{5.9}
\frac{p(z,w)z^n\pb(1/z,\wb_1)-\pr(z,w) \wb_1^{m}p(z,1/\wb_1)}{1-w\wb_1}=\Phi_{n,m-1}(z,w)^{T}\,\Phir^{m-1}_{n}(z)^{T}
\left[\begin{matrix} \wb_1^{m-1} \\ \wb_1^{m-2} \\ \vdots \\1\end{matrix}\right].
\end{equation}
Let us denote by $S(z,w,\wb_1)$ the function on the left-hand side above. Note that we can rewrite $S(z,w,\wb_1)$ as follows
\begin{equation}\label{5.10}
S(z,w,\wb_1)=p(z,w)\frac{A(z,1/w,\wb_1)}{w}+\pr(z,w)\frac{B(z,1/w,\wb_1)}{w},
\end{equation}
where
\begin{equation*}
A(z,1/w,\wb_1)=\frac{z^{n}\pb(1/z,\wb_1)-z^{n}\pb(1/z,1/w)}{1/w-\wb_1}
\end{equation*}
and
\begin{equation*}
B(z,1/w,\wb_1)=\frac{w^{-m}p(z,w)-\wb_1^{m}p(z,1/\wb_1)}{1/w-\wb_1}
\end{equation*}
are polynomials in $z$, $1/w$ and $\wb_1$ of degrees at most $n$, $m-1$ and $m-1$, respectively. Thus, there exist 
$1\times m$ vectors $A_m(z,1/w)$ and $B_m(z,1/w)$ whose entries are polynomials in $z$ and  $1/w$ of degrees at most $n$ and $m-1$, respectively, such that 
$$A(z,1/w,\wb_1)=A_m(z,1/w)\left[\begin{matrix} \wb_1^{m-1} \\ \wb_1^{m-2} \\ \vdots \\1\end{matrix}\right]
\text{ and }
B(z,1/w,\wb_1)=B_m(z,1/w)\left[\begin{matrix} \wb_1^{m-1} \\ \wb_1^{m-2} \\ \vdots \\1\end{matrix}\right].$$
Combining the last equation with equations \eqref{5.9} and \eqref{5.10} we see that 
\begin{equation*}
p(z,w)\frac{A_m(z,1/w)}{w}-\pr(z,w)\frac{B_m(z,1/w)}{w} =\Phi_{n,m-1}(z,w)^{T}\,\Phir^{m-1}_{n}(z)^{T}.
\end{equation*}
From the theory of matrix-valued orthogonal polynomials we know that $\det(\Phir^{m-1}_{n}(z))\neq0$ for $|z|\leq1$. Therefore, 
the entries of the matrix $[\Phir^{m-1}_{n}(z)^{T}]^{-1}$ are analytic functions on closed unit disk $|z|\leq 1$ and we have
\begin{equation*}
\Phi_{n,m-1}(z,w)^{T}=\left(p(z,w)\frac{A_m(z,1/w)}{w}-\pr(z,w)\frac{B_m(z,1/w)}{w}\right)\left[\Phir^{m-1}_{n}(z)^{T}\right]^{-1}.
\end{equation*}
Equation \eqref{5.8} follows immediately from the last equation and \leref{le5.1}.
\end{proof}

\begin{proof}[Proof of \thref{th2.7}]
The block structure of the vector polynomials given in equation \eqref{2.18} follows immediately from equation \eqref{5.7} in \leref{le5.2}. 
Let us denote by $\cLp$ and $\cLq$ the positive moment functionals corresponding to the stable polynomials $p(z,w)$ and $q(z,w)$, i.e. 
\begin{equation*}
\cLp(z^kw^l)=\frac{1}{4\pi^2}\int\limits_{[-\pi,\pi]^2}\frac{e^{ik\theta}e^{il\varphi}}{|p(e^{i\theta},e^{i\varphi})|^2}\,d\theta\,d\varphi
\end{equation*}
and
\begin{equation*}
\cLq(z^kw^l)=\frac{1}{4\pi^2}\int\limits_{[-\pi,\pi]^2}\frac{e^{ik\theta}e^{il\varphi}}{|q(e^{i\theta},e^{i\varphi})|^2}\,d\theta\,d\varphi.
\end{equation*}
To prove that there exists a unitary matrix $U$ such that \eqref{2.19a} holds, 
it is enough to show two things:
\begin{itemize}
\item[{(i)}] the entries of the vector polynomial on the right-hand side of \eqref{2.19a} form an orthonormal set of polynomials of degrees 
at most $(n,m-1)$ with respect to $\cL$; 
\item[{(ii)}] the entries of the vector polynomial on the right-hand side of \eqref{2.19a} are orthogonal with respect to $\cL$ to all 
polynomials of degrees at most $(n-1,m-1)$. 
\end{itemize}
Clearly, the entries of the vector polynomial on the right-hand side of \eqref{2.19a} are polynomials of degrees at most $(n,m-1)$ and it is easy to see that they all have norm $1$. The fact that they are mutually orthogonal follows from the following direct computation
\begin{align*}
&\cL(z^{-n_2}\qb(z,1/w)\overline{\Phip_{n_1,m_1-1}}(1/z,1/w)\,\pr(z,w)w^{m_2-1}\overline{\Phiq_{n_2,m_2-1}}(z,1/w)^{T})\\
&\qquad
=-\frac{1}{4\pi^2}\int\limits_{\Tset}\left[\int\limits_{\Tset}\frac{\Phir^{p}_{n_1,m_1-1}(z,w)^{T}\, w^{m_2-1}\overline{\Phiq_{n_2,m_2-1}}(z,1/w)^{T}}
{p(z,w)z^{n_2}q(1/z,w)}\,dw\right]\frac{dz}{z}\\
&\qquad=0,
\end{align*}
since the $w$-integral is zero by Cauchy's residue theorem. Thus, it remains to check (ii). Below we compute the inner products with the monomials $z^{k}w^{l}$ where $0\leq k\leq n-1=n_1+n_2-1$ and $0\leq l\leq m-1=m_1+m_2-1$. 

For the first $m_1$ entries on the right-hand side of  \eqref{2.19a} we obtain
\begin{align*}
&\cL(z^{n_2}q(1/z,w)\Phip_{n_1,m_1-1}(z,w) z^{-k}w^{-l})\\
&\qquad=-\frac{1}{4\pi^2}\int\limits_{\Tset^2}\frac{\Phip_{n_1,m_1-1}(z,w) z^{n_2-k}w^{-l}}{|p(z,w)|^2\qb(z,1/w)}\frac{dz}{z}\frac{dw}{w}\\
&\qquad=\cLp\left(\Phip_{n_1,m_1-1}(z,w) \;\frac{z^{n_2-k}}{w^{l}\qb(z,1/w)}\right)=0,
\end{align*}
by equation \eqref{5.8} in \leref{le5.2} and the computation for the last $m_2$ entries on the right-hand side of  \eqref{2.19a} is similar. 
The fact that there exists a unitary matrix $V$ such that \eqref{2.19b} holds 
can be established along the same lines.
\end{proof}

\begin{Remark}\label{re5.3}
The decomposition in \thref{th2.7} can be naturally connected to a decomposition of a Christoffel-Darboux type formula. Indeed, for a polynomial $h(z,w)$ let us consider the corresponding  Christoffel-Darboux kernel
\begin{equation*}
L^{h}(z,w;\eta) = \frac{h(z,w)\overline{h(1/\zb,\eta)} - \hr(z,w)\overline{\hr(1/\bar{z},\eta)}}{1-w\bar{\eta}}.
\end{equation*}
Using the notations in \thref{th2.7}, we set $q^{w}(z,w)=z^{n_2}q(1/z,w)$ and $h(z,w)=p(z,w)q^{w}(z,w)$. Then it is easy to see that 
$$L^{h}(z,w;\eta)=q^{w}(z,w)\overline{q^{w}(1/\zb,\eta)}L^{p}(z,w;\eta)+\pr(z,w)\overline{\pr(1/\zb,\eta)}L^{q^{w}}(z,w;\eta).$$
The point now is that the polynomials $p(z,w)$ and $q(z,w)$ are stable for $|z|\leq 1$, $|w|\leq 1$ and therefore the corresponding kernels possess a great many orthogonality relations (see for instance \cite{GIK}) which can be used to prove equations \eqref{2.19} and thus give an alternate proof of the implication (i)$\Rightarrow$(ii) in \thref{th2.5}. 

It is a challenging problem to find a direct algebro-geometric proof of the implication (i)$\Rightarrow$(ii) in \thref{th2.3} (i.e. if we have stability only with respect to one of the variables). If we use the notations in \thref{th2.3} and \leref{le4.4}, the heart of the problem is the following:  start with a polynomial $p(z,w)$ which is stable for $|z|=1$ and $|w|\leq 1$ and give an explicit description (or prove existence) of the spaces $H_1$ and $H_2$, where $H_j$ is the space spanned by the entries of the vector polynomials $\Psit^{(j)}_{n-1,m}(z,w)$ in \leref{le4.4}. These spaces must be mutually orthogonal and must satisfy additional extra orthogonality properties in view of equation \eqref{4.10b}. Since the orthogonality relations can be expressed in terms of residues, constructing bases for these spaces amounts to an interesting interpolation problem on a zero-dimensional variety, which involves appropriate zeros of $p(z,w)$ and $\pr(z,w)$. Equivalently, this would
  give a subtle decomposition of the Christoffel-Darboux kernel associated with $p(z,w)$.
\end{Remark}

\subsection{Proofs of Corollaries \ref{co2.8} and \ref{co2.9}}\label{ss5.3}
\begin{proof}[Proof of \coref{co2.8}]
The statement in (i) is proved in \cite[Theorem 7.2]{GW2} but we sketch it briefly below since it follows easily from the constructions in this paper. If \eqref{2.20} holds then the defining relation \eqref{2.12c} for $\cK_{n,m}$ and \leref{le5.2} show that $\cK_{n,m}=0$. Conversely, suppose that $\cK_{n,m}=0$. Equation \eqref{2.11c} and its tilde analog imply that 
\begin{equation}\label{5.11}
\Phi_{n,m}(z,w)=\left[\begin{matrix} \phi_{n,m}^{m}(z,w)\\[-6pt] \\ \Phi_{n,m-1}(z,w) \end{matrix}\right], \quad 
\Phit_{n,m}(z,w)=\left[\begin{matrix} \phit_{n,m}^{n}\\[-6pt]\\ \Phit_{n-1,m}(z,w) \end{matrix}\right].
\end{equation}
Using the second equation above and \leref{le4.4} (with $\Ut$ and $\Vt$ being the identity matrices), we see that equation \eqref{2.20} holds where $p(z,w)=\phir_{n,m}^{m}(z,w)$ is stable for $|z|=1$ and $|w|\leq1$. Since $\phi_{n,m}^{m}(z,w)=\phit_{n,m}^{n}(z,w)$ we can use the first equation in \eqref{5.11} and the tilde analog of \leref{le4.4} to deduce that $p(z,w)$ is stable also for $|z|\leq 1$ and $|w|=1$, which shows that $p(z,w)$ is stable for $|z|\leq 1$ and $|w|\leq 1$ completing the proof of (i).\\

Suppose now that \eqref{2.21} holds. Applying \thref{th2.7} and its tilde analog we see that there exist unitary matrices $U\in M^{m,m}$, $\Ut\in M^{n,n}$ such that 
\begin{equation*}
\Phi_{n,m-1}(z,w)=w^{m-1}U\overline{\Phiq_{n,m-1}}(z,1/w), \quad 
\Phit_{n-1,m}(z,w)=z^{n-1}\Ut\tilde{\Phi}^q_{n-1,m}(1/z,w).
\end{equation*}
Plugging these formulas in the definition \eqref{2.12e} of $\cK^{1}_{n,m}$ we find 
$$\cK^{1}_{n,m}=U\,\overline{\langle \Phiq_{n,m-1}(1/z,w) , \tilde{\Phi}^q_{n-1,m}(1/z,w)\rangle}\, \Ut^{T}.$$
Note that the inner product in the expression above gives the matrix $\cK_{n,m}$ for the measure $\frac{d\theta\,d\varphi}{4\pi^2|q(e^{i\theta},e^{i\varphi})|^2}$, and therefore is zero from the first part of the corollary. Thus, $\cK^{1}_{n,m}=0$. 

Conversely, suppose now that $\cK^{1}_{n,m}=0$. From \eqref{2.11d} and its tilde analog we see that there exist unitary matrices $V\in M^{m+1,m+1}$, $\Vt\in M^{n+1,n+1}$ such that
\begin{equation}\label{5.12}
V^{\dagger}\Phi_{n,m}(z,w)=\left[\begin{matrix} \psi_{n,m}^{m}(z,w)\\ w\Phi_{n,m-1}(z,w) \end{matrix}\right], \quad 
\Vt^{\dagger}\Phit_{n,m}(z,w)=\left[\begin{matrix} \psit_{n,m}^{n}(z,w)\\ z\Phit_{n-1,m}(z,w) \end{matrix}\right].
\end{equation}
From \leref{le4.4} we deduce that \eqref{2.21} holds with $q(z,w)=z^n\psitr_{n,m}^{n}(1/z,w)$, which is stable for $|z|=1$ and $|w|\leq1$. 
We want to show next that $\psit_{n,m}^{n}(z,w)$ and $\psir_{n,m}^{m}(z,w)$ are equal up to a unimodular constant, i.e.
\begin{equation}\label{5.13}
\psit_{n,m}^{n}(z,w)=\epsilon \psir_{n,m}^{m}(z,w), \text{ where }|\epsilon|=1.
\end{equation}
Note that if we can prove the above equation, we can use the tilde analog of  \leref{le4.4} to deduce that $q(z,w)=\bar{\epsilon}z^n\psi_{n,m}^{m}(1/z,w)$ is stable for $|z|\leq 1$ and $|w|=1$, thus proving that $q(z,w)$ is stable for $|z|\leq 1$ and $|w|\leq1$.

The proof of \eqref{5.13} follows from the characteristic properties of $\psi_{n,m}^{m}(z,w)$ and $\psit_{n,m}^{n}(z,w)$. Indeed, from the first equation in \eqref{5.12} it is easy to see that $\psi_{n,m}^{m}(z,w)$ is the unique (up to a unimodular constant) orthonormal vector in $\Pinm$ such that 
$$\psi_{n,m}^{m}(z,w)\perp\{z^kw^l:0\leq k\leq n-1,\; 0\leq l\leq m\}\cup\{z^nw^l:1\leq l\leq m\}.$$
Similarly, from the second equation in \eqref{5.12} we see that $\psit_{n,m}^{m}(z,w)$ is the unique (up to a unimodular constant) orthonormal vector in $\Pinm$ such that 
$$\psit_{n,m}^{m}(z,w)\perp\{z^kw^l:0\leq k\leq n,\; 0\leq l\leq m-1\}\cup\{z^kw^m:1\leq k\leq n\}.$$
The above characteristic properties of $\psi_{n,m}^{m}(z,w)$ and $\psit_{n,m}^{n}(z,w)$ establish \eqref{5.13}, thus completing the proof.
\end{proof}

\begin{proof}[Proof of \coref{co2.9}]
Assume first that $\cL(z^kw^l)=\cL_z(z^k)\cL_w(w^l)$. If we denote by $\{\alpha_k(z)\}_{0\leq k\leq n}$ the (one-variable) polynomials orthonormal with respect to $\cL_z$ and by $\{\beta_l(w)\}_{0\leq l\leq m}$ the (one-variable) polynomials orthonormal with respect to $\cL_w$, then it easy to see that 
\begin{equation*}
\Phi_{n,m}(z,w)=\alpha_n(z)\left[\begin{matrix} \beta_{m}(w)\\ \beta_{m-1}(w)\\[-2pt]
\vdots\\ \beta_{0}(w) \end{matrix}\right] , 
\text{ and }
\Phit_{n,m}(z,w)=\beta_m(w)\left[\begin{matrix} \alpha_{n}(z)\\ \alpha_{n-1}(z)\\[-2pt]
\vdots\\ \alpha_{0}(z) \end{matrix}\right]. 
\end{equation*}
From these explicit formulas and the defining relations \eqref{2.12c}, \eqref{2.12e} for $\cK_{n,m}$ and $\cK^{1}_{n,m}$ it easy to see that $\cK_{n,m}=\cK^{1}_{n,m}=0$.

Conversely, suppose that $\cK_{n,m}=\cK^{1}_{n,m}=0$. Note that if $h(z,w)$ is a polynomial of degree $(k,l)$ such that $h(z,w)$ and $z^{k}h(1/z,w)$ are stable, then $h(z,w)$ is independent of $z$ (i.e. $k=0$). Using this observation, \coref{co2.8} and arguments similar to the ones we used in the proof of the implication (ii)$\Rightarrow$(i) in \thref{th2.5}, we see that $\phir_{n,m}^{m}(z,w)=\alpha(z)\beta(w)$, where $\alpha(z)$ and $\beta(w)$ are stable polynomials of degrees at most $n$ and $m$, respectively and that equation  \eqref{2.22} holds, completing the proof.
\end{proof}

\subsection{Proof of \thref{th2.10}}\label{ss5.4}

With the measure $d\mu$ we will associate the positive moment functional 
$\cL$ defined on $\Cset[z,z^{-1},w,w^{-1}]$ by 
\begin{equation}\label{5.14}
\cL(z^kw^l)
=\int\limits_{\Tset^2}z^kw^ld\mu.
\end{equation} 

First suppose that equation \eqref{2.23} holds where $p(z,w)$ and 
$q(z,w)$ are stable polynomials of degrees $(n_1,m_1)$ and $(n_2,m_2)$, 
respectively, with $n_1+n_2\leq n$, $m_1+m_2\leq m$. Then we can represent 
$\cL$ as in equations \eqref{5.1} where $P(z,w)$ and $Q(z,w)$ are given in 
equations \eqref{5.2}. Using \leref{le4.2} and its tilde analog, we see that 
equation \eqref{2.25} holds. To complete the proof of the theorem, it remains 
to show that equation \eqref{2.24} implies the existence of stable polynomials 
$p(z,w)$ and $q(z,w)$ of degrees $(n_1,m_1)$ and $(n_2,m_2)$, 
with $n_1+n_2\leq n$, $m_1+m_2\leq m$ such that \eqref{2.23} holds. 
From \leref{le4.3} and its tilde analog we see that 
\begin{subequations}\label{5.15}
\begin{align}
&\cK_{k,l}\left[\Gt^{1}_{k,l}\Gt^{\dagger}_{k,l}\right]^j(\cK^{1}_{k,l})^{T}=0, 
\text{ for all }j\geq 0,\quad k\geq n+2,\quad l\geq m+2, \label{5.15a}\\
&(\cK_{k,l})^{\dagger}\left[\G^{1}_{k,l}\G^{\dagger}_{k,l}\right]^j\cK^{1}_{k,l}=0, 
\text{ for all }j\geq 0,\quad k\geq n+2,\quad l\geq m+2. \label{5.15b}
\end{align}
\end{subequations}
By \thref{th2.5} we deduce that there exist stable polynomials 
$p(z,w)$ and $q(z,w)$ of degrees $(n_1,m_1)$ and $(n_2,m_2)$, 
with $n_1+n_2\leq n+2$, $m_1+m_2\leq m+2$ such that \eqref{2.16} holds for all 
$(k,l)$ satisfying $|k|\leq n+2$, $|l|\leq m+2$. Moreover, from equations 
\eqref{2.15} we see that 
\begin{equation}\label{5.16}
\begin{split}
&p(z,w)\pb(1/z,1/w)q(1/z,w)\qb(z,1/w)\\
&=\Phit_{n+2,m+2}(z,w)^{T}\overline{\Phit_{n+2,m+2}}(1/z,1/w)
-\Phit_{n+1,m+2}(z,w)^{T}\overline{\Phit_{n+1,m+2}}(1/z,1/w) \\
&=\Phi_{n+2,m+2}(z,w)^{T}\overline{\Phi_{n+2,m+2}}(1/z,1/w)
-\Phi_{n+2,m+1}(z,w)^{T}\overline{\Phi_{n+2,m+1}}(1/z,1/w).
\end{split}
\end{equation}
Recall that if $\Eh_{k,l}=0$ then $A_{k,l}=I_{l+1}$ and therefore by 
\eqref{2.11a} we obtain 
\begin{subequations}\label{5.17}
\begin{equation}\label{5.17a}
\Phi_{k,l}(z,w)=z\Phi_{k-1,l}(z,w).
\end{equation}
Similarly, if $\Eht_{k,l}=0$ then
\begin{equation}\label{5.17b}
\Phit_{k,l}(z,w)=w\Phi_{k,l-1}(z,w).
\end{equation}
\end{subequations}
Using \eqref{2.24} we see that equations \eqref{5.17} hold for all $k\geq n+1$
and $l\geq m+1$, which combined with \eqref{5.16} shows that 
\begin{subequations}\label{5.18}
\begin{align}
&p(z,w)\pb(1/z,1/w)q(1/z,w)\qb(z,1/w)\nn\\
&=\Phit_{k,l}(z,w)^{T}\overline{\Phit_{k,l}}(1/z,1/w)
-\Phit_{k-1,l}(z,w)^{T}\overline{\Phit_{k-1,l}}(1/z,1/w) \label{5.18a}\\
\intertext{ for $k\geq n+2$, $l\geq m$, and }
&p(z,w)\pb(1/z,1/w)q(1/z,w)\qb(z,1/w)\nn\\
&=\Phi_{k,l}(z,w)^{T}\overline{\Phi_{k,l}}(1/z,1/w)
-\Phi_{k,l-1}(z,w)^{T}\overline{\Phi_{k,l-1}}(1/z,1/w)\label{5.18b}
\end{align}
\end{subequations}
for $k\geq n$, $l\geq m+2$.
From \thref{th2.3} we see that equation \eqref{2.16} holds for all 
$k,l\in\Zset$ which establishes \eqref{2.23}. It remains to show now that in 
fact  $n_1+n_2\leq n$ and $m_1+m_2\leq m$. To see this, we will use the 
following two observations:
\begin{itemize}
\item[{(i)}] If $p(z,w)$ and $q(z,w)$ are stable polynomials of degrees $(n_1,m_1)$ and $(n_2,m_2)$, then $P(z,w)=p(z,w)z^{n_2}q(1/z,w)$ is a polynomial of degree $(n_1+n_2,m_1+m_2)$ which is not divisible by $z$ and $w$ (i.e. $P(0,w)\not\equiv0$ and $P(z,0)\not\equiv0$).
\item[{(ii)}] If $P(z,w)$ is a polynomial of degree $(n_0,m_0)$, which is not 
divisible by $z$ and $w$ such that $P(z,w)\bar{P}(1/z,1/w)\in\Pinm$, then 
$n_0\leq n$ and $m_0\leq m$.
\end{itemize}
From (i) we see that the polynomial $P(z,w)=p(z,w)z^{n_2}q(1/z,w)$ is a 
polynomial of degree $(n_1+n_2,m_1+m_2)$ which is not divisible by $z$ and $w$. 
From equation \eqref{5.18a} with $k=n+2$, $l=m$ and equation \eqref{5.18b} with 
$k=n$, $l=m+2$ we see that 
$$P(z,w)\bar{P}(1/z,1/w)\in \Pi^{n+2,m}\cap\Pi^{n,m+2}=\Pinm,$$  
which combined with (ii) completes the proof.
\qed

\section{Examples}\label{se6}
We now consider some examples that exhibit the properties of the theorems
proved earlier. 
\subsection{One-sided}
Our first example  will be a polynomial of degree (2,2) that 
is stable for $|z|=1$, $|w|\le 1$. We will construct the polynomial using the 
algorithm given in \cite{GW2}. 
Setting  $u_{0,0}=1$, $u_{2,0}=1/4, u_{-1,2}=\frac{1-a^2}{1+a^2}$, 
$u_{2,2}=-\frac{\sqrt{15}(1-a^2)}{60 a}$, $u_{-2,2}=-\frac{a(1-a^2)}{(1+a^2)^2}$ 
with $(-3+\sqrt{13})/2< a< (3+\sqrt{13})/2$ and $u_{i,j}=0$, for 
$(i,j)\in\{(0,1),(1,0),(0,2),(-1,1),(1,1),(1,2),(-2,1),(2,1)\}$ we construct 
the orthogonal polynomials up to level (2,2). In this case we find using Maple or Mathematica that
$$
{\cK}_{2,2}=\frac{2\sqrt{15}(1-a^2)}{15(1+a^2)}\left[\begin{matrix}0&0\\ 2&-1\end{matrix}\right],
$$
$$
{\cK}^1_{2,2}=-\frac{\sqrt{15}(1-a^2)}{60a}\begin{bmatrix} 1 & 2 \\0&0 \end{bmatrix},
$$
$$
\G_{2,2}=\begin{bmatrix} 0&1&0\\0&0&\frac{\sqrt{3c}}{3(1+a^2)}\end{bmatrix},
$$
$$
\G^1_{2,2}=\begin{bmatrix}\frac{2\sqrt{3d}}{3\sqrt{c}}&0&\frac{\sqrt{3}(1-a^4)}{12 a\sqrt{c}}\\0&1&0\end{bmatrix},
$$
$$
\tilde{\G}_{2,2}=\begin{bmatrix}0&\frac{5\sqrt{c}}{\sqrt{e}}&\frac{8\sqrt{15}(1-a^2)^2}{15(1+a^2)\sqrt{f}}\\0&0&\frac{\sqrt{15 f}}{15(1+a^2)} \end{bmatrix},
$$
and
$$
\tilde{\G}^1_{2,2}=\begin{bmatrix} \frac{\sqrt{5d}}{2\sqrt{c}}&-\frac{(1-a^4)(1-a^2)\sqrt{e}}{4af\sqrt{c}}&\frac{\sqrt{15}g}{30a\sqrt{f}}\\0&\frac{5(1+a^2)\sqrt{c}}{2a\sqrt{e}}&\frac{4\sqrt{15}(1-a^2)^2}{15a\sqrt{f}} \end{bmatrix}
$$
where $c=14 a^2-a^4-1$, $d=11a^2-a^4-1$, $e=55+190 a^2+55 a^4$, $f=11+38 a^2+11 a^4$, $g=4a^4+7a^2+4$.

It is not difficult to see that ${\cK}_{2,2}({\cK}^1_{2,2})^T=0={\cK}_{2,2}^{\dag}{\cK}_{2,2}^1$, and using Maple or Mathematica we find 
$${\cK}_{2,2}\tilde{\G}^1_{2,2}\tilde{\G}_{2,2}^{\dag}({\cK}^1_{2,2})^T=0$$ 
while 
\begin{equation}\label{kdg1gk}
{\cK}_{2,2}^{\dag}{\G}^1_{2,2}{\G}_{2,2}^{\dag}{\cK}^1_{2,2}=\frac{(1-a^2)^2}{15a(1+a^2)}\begin{bmatrix}-1&-2\\\frac{1}{2}&1\end{bmatrix}\ne0.
\end{equation}
Using \reref{re4.5} we find a candidate for $p(z,w)$ is 
\begin{align*}
\hat{p}(z,w)=&(4a(1-a^2)w^2-3(1+a^2)^2)z^2+3((1-a^4)w^2+3a(1+a^2))z\\
&\qquad-13a(1-a^2)w^2+12a^2.
\end{align*}
Using the Schur-Cohn test it is not difficult to see that $\hat{p}(z,w)$ is nonzero for $|z|=1$ and $|w|\le1$.
Applying \leref{le4.6} (see also \reref{re4.7}) we see that equations \eqref{4.23} and \eqref{4.24} hold with $r=r^{1}=1$, $s_1=\frac{2(1-a^2)}{\sqrt{3}(1+a^2)}$, $s^{1}_1=\frac{1-a^2}{4\sqrt{3}a}$,
$$
U^{1}=U=\begin{bmatrix}0&1\\1&0\end{bmatrix},
\text{ and }
\Ut=\frac{1}{\sqrt{5}}\begin{bmatrix}2&-1\\-1&-2\end{bmatrix}.$$
Next we find $\Psit_{1,2}(z,w)$ in \eqref{4.10a} by computing
$$
\Psit_{1,2}(z,w)=\Ut^{\dagger}\Phit_{1,2}(z,w)=\frac{\sqrt{5}}{10a}\begin{bmatrix}4azw^2-w^2-a^2 w^2+z-a^2 z\\
 -2azw^2-2w^2-2a^2w^2+2z-2a^2 z\end{bmatrix}.
$$
We look for a unitary matrix $\Vt$ such that equation \eqref{4.10b} holds with 
$n_1=n_2=1$.

This uniquely specifies $\tilde V$ (except the first column, which can be multiplied by an arbitrary complex number of modulus $1$) as
$$
\Vt=\begin{bmatrix} \frac{\sqrt{3}a}{\sqrt{c}} & \frac{\sqrt{d}}{\sqrt{c}} & 0\\
\frac{2\sqrt{3}(1+a^2)\sqrt{d}}{\sqrt{cf}} & -\frac{6a(1+a^2)}{\sqrt{fc}}&-\frac{\sqrt{c}}{\sqrt{f}}\\
-\frac{\sqrt{d}}{\sqrt{f}} & \frac{\sqrt{3}a}{\sqrt{f}} & -\frac{2\sqrt{3}(1+a^2)}{\sqrt{f}}
\end{bmatrix}.
$$
The first entry of the vector polynomial $\Psit_{2,2}(z,w)=\Vt^{\dagger}\Phit_{2,2}(z,w)$  is $\frac{\sqrt{15}}{30a\sqrt{d}}\overleftarrow{\hat p}(z,w)$. Thus, \leref{le4.4} shows that $p(z,w)=\frac{\sqrt{15}}{30a\sqrt{d}}\hat p(z,w)$.  

\subsection{Splitting case}

The second example we will consider illustrates \thref{th2.5}.
In this case we chose $u_{0,0}=1$, $u_{-1,1}=a$, $u_{1,1}=b$, $u_{2,0}=ab=u_{0,2}$
and $u_{i,j}=0,\ (i,j)\in\{ (0,1),(1,0),(1,2),(-1,2),(2,1),(-2,1),(2,2),(-2,2)\}$ where $-1<a<1$ and $-1<b<1$. Using the algorithm given in \cite{GW2} we find
that
$$
{\cK}_{2,2}=\begin{bmatrix}a&0\\0&0 \end{bmatrix},
$$
$$
{\cK}^{1}_{2,2}=\begin{bmatrix}0&0\\0&b \end{bmatrix},
$$
$$
\G_{2,2}=\begin{bmatrix}0&\sqrt{1-a^2}&0\\0&0&1 \end{bmatrix}=\tilde{\G}_{2,2},
$$
$$
\G^1_{2,2}=\begin{bmatrix}1&0&0\\0&\sqrt{1-b^2}&0 \end{bmatrix}=\tilde{\G}^1_{2,2}.
$$
It is easy to check that equations~\eqref{2.17} are satisfied. Moreover, we 
find that
\begin{equation}\label{PPP}
P(z,w)\bar{P}(1/z,1/w)=\Phit_{2,2}(z,w)^{T}\overline{\Phit}_{2,2}(1/z,1/w)
-\Phit_{1,2}(z,w)^{T}\overline{\Phit}_{1,2}(1/z,1/w)
\end{equation}
where $$P(z,w)=\frac{(1-bzw)(z-aw)}{\sqrt{(1-a^2)(1-b^2)}}$$
is stable for $|z|=1$ and $|w|\leq 1$. It is easy to see that the polynomial 
$P(z,w)$ above is the unique polynomial (up to a multiplicative constant of 
modulus 1) of degree at most $(2,2)$ which is stable for $|z|=1$ and 
$|w|\leq 1$ and which satisfies \eqref{PPP}. Finally, note that
$$P(z,w)=p(z,w)zq(1/z,w)$$
where
$$p(z,w)=\frac{1-bzw}{\sqrt{1-b^2}}, \quad \text{ and }\quad q(z,w)=\frac{1-azw}{\sqrt{1-a^2}}$$ 
are stable polynomials. 
We can obtain all this also by following the steps of Example 1. Indeed, we see that we can take $U=\Ut=I_2$ the identity 
$2\times 2$ matrix and
$$\Vt=\begin{bmatrix}0&1&0\\1&0&0\\0&0&1\end{bmatrix}.$$
The first entry of the vector polynomial $\Psit_{2,2}(z,w)=\Vt^{\dagger}\Phit_{2,2}(z,w)$ is $P(z,w)$. Note that 
if $a=0$ then ${\cK}_{2,2}=0$ and the functional is in the stable case, 
while if $b=0$ then ${\cK}^1_{2,2}=0$ and the functional is in the 
anti-stable case.

\end{document}